\documentclass[article]{article}
\usepackage{color}
\usepackage{overpic}
\usepackage{graphicx}
\usepackage{amsmath}
\usepackage{amssymb}
\usepackage{mathrsfs}
\usepackage[numbers,sort&compress]{natbib}
\usepackage{amsthm}
\usepackage{tikz}
\usetikzlibrary{fit}
\usepackage{float}
\usepackage[T1]{fontenc}
\usepackage{hyperref}
\hypersetup{hypertex=true,
            colorlinks=true,
            linkcolor=blue,
            anchorcolor=blue,
            citecolor=blue}
\makeatletter

\newtheorem{proposition}{Proposition}[section]

\newtheorem{property}{Assumption}

\newtheorem{thm}{Theorem}
\@addtoreset{equation}{section}
\makeatother

\begin{document}
\title{Soliton resolution for the Harry Dym equation with 
weighted Sobolev initial data}
\author{Lin Deng$^1$,  Zhenyun Qin$^1$\thanks{\ Corresponding author and email address: zyqin@fudan.edu.cn } }
\footnotetext[1]{ \  School of Mathematical Sciences, Fudan University, Shanghai 200433, P.R. China.}

\date{ }

\maketitle
\begin{abstract}
\baselineskip=16pt

The soliton resolution for the Harry Dym equation is established for initial conditions in weighted Sobolev space $H^{1,1}(\mathbb{R})$. Combining the nonlinear steepest descent method and $\bar{\partial}$-derivatives condition, we obtain that when $\frac{y}{t}<-\epsilon(\epsilon>0)$ the long time asymptotic expansion of the solution $q(x,t)$ in any fixed cone 
\begin{equation}
   C\left(y_{1}, y_{2}, v_{1}, v_{2}\right)=\left\{(y, t) \in R^{2} \mid y=y_{0}+v t, y_{0} \in\left[y_{1}, y_{2}\right], v \in\left[v_{1}, v_{2}\right]\right\}
\end{equation}
up to an residual error of order $\mathcal{O}(t^{-1})$. The expansion shows the long time asymptotic behavior can be described as an $N(I)$-soliton on discrete spectrum whose parameters are modulated by a sum of localized soliton-soliton interactions as one moves through the cone and the second term coming from soliton-radiation interactionson on continuous spectrum.
\\
\\Keywords: The Harry Dym equation; Riemann-Hilbert problem;  $\bar{\partial}$ steepest descent method; Soliton resolution.
\end{abstract}
\baselineskip=16pt

\newpage
\tableofcontents
\section{Introduction}
\hspace*{\parindent}
The Harry Dym equation was first discovered by H.Dym in 1973-1974, and its first appearance in the literature occurred in a 1975 paper of Kruskal \cite{MD1975} where it was named after its discoverer, then it was rediscovered in more general form in \cite{sabatier1979} within the classical string problem. In 1978, the Harry Dym equation, the KdV equation, the mKdV equation and the nonlinear Schr\"{o}dinger equation were used as examples to verify a general model of integrable Hamiltonian equation \cite{magri1978}, which gives its bi-Hamiltonian structure. The equation is proved to satisfy the properties: infinite number of conservation laws and infinitely many symmetries \cite{PRD1983}. The recursion operators and complete Lie-B\"{a}cklund symmetry of this equation were given by \cite{Chowdhury1984}.

Direct links were found between the KdV equation and Harry Dym equation \cite{RC1986}, or mKdV equation and Harry Dym equation \cite{Ks1985,Dmitrieva1994}. The Lax pair of the Harry Dym equation is associated with the Sturm-Liouville operator. The Liouville transformation transforms this operator isospectrally into the Schr\"{o}dinger operator \cite{GESZTESY1992}. The Dym equation is an important nonlinear partial differential equation which is integrable and finds applications in several physical systems, which is related to such physical problems as the Hele-Shaw problem \cite{howison_1992}, the Saffmae-Taylor problem and the chiral dynamics of closed curves on the complex plane \cite{PhysRevLett1990}.

Harry Dym equation, as a completely integrable non-linear evolution equation, admits a cusp solitary wave solution in implicit expression by the inverse scattering transform (IST) in \cite{wm1980}. It has various forms of solutions: the algebraic geometric solution \cite{novikov_1999}, the exact solution \cite{Mokhtari_2011}, the elliptic solution \cite{Chowdhury1984} and N-loop solitons are constructed by means of N-cusp soliton solutions of Harry Dym equation \cite{MD1975}.

In 1974, Manakov first studied the long time behavior of the nonlinear wave equation solvable by the inverse scattering method \cite{Manakov1974}. Then Zakharov and Manakov use this method and give the first result of the large time asymptoticity of the NLS equation \cite{Zakharov1976}. The inverse scattering method also apply to KdV, Landau-Lifshitz and the reduced Maxwell-Bloch system \cite{Schuur1986,fokas1992,bikbaev1988}. In 1993, Deift and Zhou proposed the nonlinear steepest descent method which can obtain the long-term asymptotic behavior of the solution for the MKdV equation by deforming contours to reduce the original Riemann-Hilbert problem (RHP) to a model one whose solution is calculated in terms of parabolic cylinder functions \cite{deift1993}. Since then, this method has been widely used for focusing NLS equation \cite{deift1994}, KdV equation \cite{grunert2009}, Fokas-Lenells equation \cite{xu2015}, derivative NLS equation \cite{xu2013}, short pulse equation \cite{xu2018}, Camassa-Holm equations \cite{de2009} and the Harry Dym equation \cite{xiao2019}.

In recent years, McLaughlin and Miller further presented a $\bar{\partial}$ steepest descent method which combine steepest descent with $\bar{\partial}$-problem rather than the asymptotic analysis of
singular integrals on contours to analyze asymptotic of orthogonal polynomials with non-analytical weights \cite{mclaughlin2008}. When it is applied to integrable systems, the $\bar{\partial}$ steepest descent
method also has displayed some advantages, such as avoiding delicate estimates involving $L^p$ estimates of Cauchy projection operators, and leading the non-analyticity in the RHP reductions to a $\bar{\partial}$-problem in some sectors of the complex plane which can be solved by being recast into an integral equation and by using Neumann series. Dieng and McLaughin use it to study the defocusing NLS equation under essentially minimal regularity assumptions on finite mass initial data \cite{dieng2008}. This $\bar{\partial}$ steepest descent method also was successfully applied to prove asymptotic stability of N-soliton solutions to focusing NLS equation \cite{Dbar2018}. Jenkins studied soliton resolution for the derivative nonlinear NLS equation for generic initial data in a weighted Sobolev space \cite{jenkins2018}. Their work provided the soliton resolution property for derivative NLS equation, which decomposes the solution into the sum of a finite number of separated solitons and a radiative parts when $t\rightarrow\infty$. And the dispersive part contains two components, one coming from the continuous spectrum and another from the interaction of the discrete and continuous spectrum.

In this paper, the main purpose is to study the long time asymptotic behavior for the initial value problem of the Harry Dym equation, we apply $\bar{\partial}$ steepest descent method \cite{Dbar2018} to study the Cauchy problem for the equation
\begin{equation}
q_{t}-2\left(\frac{1}{\sqrt{1+q}}\right)_{x x x}=0 \label{HDe1}
\end{equation}
with the initial value
\begin{equation}
q(x, 0)=q_{0}(x)\in H^{1,1}(\mathbb{R}),\label{chuzhi}
\end{equation}
where
\begin{equation}
H^{1,1}(\mathbb{R})=\left\{f \in L^{2}(\mathbb{R}): f^{\prime}, x f \in L^{2}(\mathbb{R})\right\}
\end{equation}
and $q(x, t) \geq-1$ for $x \in R, t \geq 0.$ 

In general, the matrix Riemann-Hilbert problem is defined in the $\lambda$ plane and has explicit $(x, t)$ dependence, while for the Harry-Dym equation (\ref{HDe1}), we need to construct a new matrix Riemann-Hilbert problem with explicit $(y, t)$ dependence, where $y(x, t)$ is a function unknown from the initial value condition. For this purpose, let $u=\sqrt{1+q}$, the problem of Harry Dym equation (\ref{HDe1}) transforms into
\begin{equation}
\left(u^{2}\right)_{t}-2\left(\frac{1}{u}\right)_{x x x}=0,\label{HDe2}
\end{equation}
$$u(x, 0)=u_{0}(x)=\sqrt{1+q_{0}(x)}.$$
In addition, assume that the initial value $q_0(x)$ satisfy two conditions:
\begin{equation}
\int_{-\infty}^{+\infty} q_{0}(x) d x=0, \quad \int_{-\infty}^{+\infty} \int_{\pm \infty}^{x} q_{0}\left(x^{\prime}\right) d x^{\prime} d x=0, \quad x \in R.\label{assum1}
\end{equation}
The assumptions in (\ref{assum1}) imply the following conditions which are needs for the spectral analysis:
\begin{equation}
\int_{-\infty}^{+\infty} q(x, t) d x=0, \quad \int_{-\infty}^{+\infty} \int_{\pm \infty}^{x} q\left(x^{\prime}, t\right) d x^{\prime} d x=0, x \in R, t \geq 0.
\end{equation}

\textbf{Organization of the paper:} In section 2, by deformation of Lax pairs, we analyze eigenfunctions at spectral parameter $\lambda\rightarrow\infty$ and $\lambda\rightarrow0$ with the initial data. In section 3, the solution of (\ref{HDe1}-\ref{chuzhi}) can be expressed by associated matrix RHP for $M(\lambda)$ dependent on the new parameter $(y, t)$ when $\lambda\rightarrow0$. In section 4, by introducing a function $T(\lambda)$, we get a new RHP for $M^{(1)}(\lambda)$, which admits a regular discrete spectrum distribution and two triangular decompositions of the jump matrix near the phase point $\pm\lambda_0$. In section 5, we construct a mixed $\bar{\partial}$-RH problem for $M^{(2)}(\lambda)$ by introducing a transformation caused by the function $\mathcal{R}(\lambda)$, which make continuous extension off the real axis to the jump matrix. In section 6, we decompose $M^{(2)}(\lambda)$ into two parts: $M^{(2)}_{RHP}(\lambda)$ and a pure $\bar{\partial}$ problem for $M^{(3)}(\lambda)$. And $M^{(2)}_{RHP}(\lambda)$ can be divided into outer model $M^{(2)}_{out}$ be solved in section 7 and the solvable RH model $M^{(\pm\lambda_0)}$  which can be approximated by a solvable model in \cite{xiao2019} in section 8. In section 9, we study the small norm problem of $E(\lambda)$. In section 10, we analyze the pure $\bar{\partial}$ problem for $M^{(3)}$. Finally, based on the series of transformations we have done, we obtain (\ref{solu}), which contributes to the soliton resolution and long-time asymptotic behavior for the Harry Dym equation.
\section{Spectral analysis}
\hspace*{\parindent}
Let $$
\sigma_{1}=\begin{pmatrix}
0 & 1 \\
1 & 0
\end{pmatrix}, \quad \sigma_{2}=\begin{pmatrix}
0 & -i \\
i & 0
\end{pmatrix}, \quad \sigma_{3}=\begin{pmatrix}
1 & 0 \\
0 & -1
\end{pmatrix}.$$
The Harry Dym equation (\ref{HDe2}) admits the following Lax pair
\begin{equation}
\left\{\begin{aligned}
&\psi_{x x}=-\lambda^{2} u^{2} \psi, \\
&\psi_{t}=2 \lambda^{2}\left[\frac{2}{u} \psi_{x}-\left(\frac{1}{u}\right)_{x} \psi\right].
\end{aligned}\right. \label{Lax pair 1}
\end{equation}
Write $\varphi=\left(\begin{array}{c}\psi \\ \psi_{x}\end{array}\right)$, then (\ref{Lax pair 1}) can be written as 
\begin{equation}
\left\{\begin{aligned}
&\varphi_{x}=M \varphi, \\
&\varphi_{t}=N \varphi.
\end{aligned}\right. \label{Lax pair 2}
\end{equation}
where 
$$
M=\begin{pmatrix}
0 & 1 \\[6pt]
-\lambda^{2} u^{2} & 0
\end{pmatrix},\
N=\begin{pmatrix}
-2 \lambda^{2}\left(\frac{1}{u}\right)_{x} & 4 \lambda^{2} \frac{1}{u} \\[6pt]
-4 \lambda^{4} u-2 \lambda^{2}\left(\frac{1}{u}\right)_{x x} & 2 \lambda^{2}\left(\frac{1}{u}\right)_{x}
\end{pmatrix}.
$$
we analyze the eigenfunctions at spectral parameter $\lambda \rightarrow 0$ and $\lambda \rightarrow \infty$ with the initial data.
\noindent
\textbf{(1) The case of $\lambda=0$}
\hspace*{\parindent}

The Harry Dym equation (\ref{HDe2}) has also the compatibility condition of the following Lax pair
\begin{equation}
\left\{\begin{aligned}
&\Phi_{x}^{0}+i \lambda \sigma_{3} \Phi^{0}=U_{0} \Phi^{0}, \\
&\Phi_{t}^{0}+4i \lambda^{3} \sigma_{3} \Phi^{0}=V_{0} \Phi^{0},
\end{aligned}\right. \label{Lax pair 3}
\end{equation}
where 
$$
\begin{aligned}
&\Phi^{0}=P\varphi,\quad P^{-1}=\begin{pmatrix}
1 & 1 \\
-i \lambda & i \lambda
\end{pmatrix},\\
&U_{0}(x, t)=\frac{\lambda q}{2}(\sigma_{2}-i \sigma_{3}), \\
&V_{0}(x, t, \lambda)=2 \lambda^{3}[\frac{q}{u} \sigma_{2}-i(\frac{1}{u}+u-2) \sigma_{3}]-2 \lambda^{2}(\frac{1}{u})_{x} \sigma_{1}+(\frac{1}{u})_{x x}(\sigma_{2}-i \sigma_{3}).
\end{aligned}
$$

Introducing the following transformation
$$\Psi^{0}=\Phi^{0}\exp[i(\lambda x+4 \lambda^{3} t) \sigma_{3}],$$
then we get the Lax pair of $\Psi^{0}$
\begin{equation}
\left\{\begin{aligned}
&\Psi_{x}^{0}+i \lambda[\sigma_{3}, \Psi^{0}]=U_{0} \Psi^{0}, \\
&\Psi_{t}^{0}+4i \lambda^{3}[\sigma_{3}, \Psi^{0}]=V_{0} \Psi^{0}.
\end{aligned}\right.\label{Lax pair 4}
\end{equation}
As $|x| \rightarrow \infty, U_{0} \rightarrow 0$. The two eigenfunctions $\Psi_{\pm}^{0}$ of equation (\ref{Lax pair 4}) satisfy the Volterra integral equations
$$
\Psi_{+}^{0}(x, t, \lambda)=I-\int_{x}^{\infty} e^{-i \lambda\left(x-x^{\prime}\right) \hat{\sigma}_{3}} U_{0}\left(x^{\prime}, t\right) \Psi_{+}^{0}\left(x^{\prime}, t, \lambda\right) d x^{\prime},
$$
$$
\Psi_{-}^{0}(x, t, \lambda)=I+\int_{-\infty}^{x} e^{-i \lambda\left(x-x^{\prime}\right) \hat{\sigma}_{3}} U_{0}\left(x^{\prime}, t\right) \Psi_{-}^{0}\left(x^{\prime}, t, \lambda\right) d x^{\prime}.
$$
Denote these vectors with superscripts $(i)$ to indicate the $i-th$ column of vectors, then
$$\Psi_{\pm}^{0} = (\Psi_{\pm}^{0,(1)}, \Psi_{\pm}^{0,(2)}). $$
From the above definition, we find that 
$$\Psi_{+}^{0,(2)}(x, t, \lambda), \Psi_{-}^{0,(1)}(x, t, \lambda) \text{ are bounded and analytic for } \lambda \in \mathbb{C}_{+},$$
$$\Psi_{+}^{0,(1)}(x, t, \lambda), \Psi_{-}^{0,(2)}(x, t, \lambda) \text{ are bounded and analytic for } \lambda \in \mathbb{C}_{-},$$
where 
$\mathbb{C}_{+}=\{\lambda \in \mathbb{C} \mid \operatorname{Im} \lambda>0\}, \mathbb{C}_{-}=\{\lambda \in \mathbb{C} \mid \operatorname{Im} \lambda<0\}.$

According to the assumptions in (\ref{assum1}) and (\ref{Lax pair 4}), $\Psi_{\pm}^{0}(x, t, \lambda)$ have the expansions in powers of $\lambda$, for $\lambda=0$,
\begin{equation}
\Psi_{\pm}^{0}(x, t, \lambda)= I+\frac{\lambda}{2} A(x, t)\left(\sigma_{2}-i \sigma_{3}\right)-\lambda^{2} A_{1}(x, t) \sigma_{1}+O\left(\lambda^{3}\right),\label{expan1}
\end{equation}
where
$$
A(x, t)=\int_{\pm \infty}^{x} q\left(x^{\prime}, t\right) d x^{\prime},\quad A_{1}(x, t)=\int_{\pm \infty}^{x} \int_{\pm \infty}^{x^{\prime}} q\left(x^{\prime \prime}, t\right) d x^{\prime \prime} d x^{\prime}.
$$
\noindent
\textbf{(2) The case of $\lambda=\infty$}
\hspace*{\parindent}

Introducing the gauge transformation
\begin{equation}
\Phi=\mathrm{PD} \varphi,
\end{equation}
where
$$
P^{-1}=\begin{pmatrix}
1 & 1 \\
-i \lambda & i \lambda
\end{pmatrix}, \quad D=\begin{pmatrix}
\sqrt{u} & 0 \\
0 & \frac{1}{\sqrt{u}}
\end{pmatrix}.
$$
Then we have the Lax pair of $\varphi$ (\ref{Lax pair 2}) becomes
\begin{equation}
\left\{\begin{aligned}
&\Phi_x+i \lambda u \sigma_{3} \Phi=U \Phi, \\
&\Phi_t+i\left[\lambda\left[\frac{1}{u}\left(\frac{1}{u}\right)_{x x}-\frac{1}{2}\left(\left(\frac{1}{u}\right)_{x}\right)^{2}\right]+4 \lambda^{3}\right] \sigma_{3} \Phi=V \Phi.
\end{aligned}\right. \label{Lax pair 5}
\end{equation}
where
$$
\begin{aligned}
&U(x, t)=\frac{1}{4} \frac{q_{x}}{u^{2}} \sigma_{1}, \\
&V(x, t, \lambda)=-2 \lambda^{2}\left(\frac{1}{u}\right)_{x} \sigma_{1}+\lambda\left[\frac{1}{u}\left(\frac{1}{u}\right)_{x x} \sigma_{2}-\frac{i}{2}\left(\left(\frac{1}{u}\right)_{x}\right)^{2} \sigma_{3}\right]+\frac{1}{2u^{2}}\left(\frac{1}{u}\right)_{x x x} \sigma_{1}. 
\end{aligned}
$$
It is clear that as $|x| \rightarrow \infty, U(x, t) \rightarrow 0, V(x, t, \lambda) \rightarrow 0$.

In order to dedicate solutions via integral Volterra equation more conveniently, we make a transformation
\begin{equation}
\Psi=\Phi \exp \left[i \lambda p(x, t, \lambda) \sigma_{3}\right],
\end{equation}
where
\begin{equation}
p(x, t, \lambda)=x+\int_{x}^{\infty}(1-u(\xi, t)) d \xi+4 \lambda^{2} t.\label{p(x,t)}
\end{equation}
From (\ref{p(x,t)}), we have 
$$
p_{x}=u(x, t),\quad p_{t}=-\int_{x}^{\infty} u_{t}(\xi, t) d \xi+4 \lambda^{2},
$$
and the conservation law
$$
u_{t}-\left(-\frac{1}{2}\left(\left(\frac{1}{u}\right)_{x}\right)^{2}+\frac{1}{u}\left(\frac{1}{u}\right)_{x x}\right)_{x}=0
$$
implies that
$$
p_{t}=-\frac{1}{2}\left(\left(\frac{1}{u}\right)_{x}\right)^{2}+\frac{1}{u}\left(\frac{1}{u}\right)_{x x}+4 \lambda^{2}.
$$
Then the Lax pair (\ref{Lax pair 5}) is changed into a new Lax pair
\begin{equation}
\left\{\begin{aligned}
&\Psi_{x}+i \lambda p_{x}\left[\sigma_{3}, \Psi\right]=U\Psi, \\
&\Psi_{t}+i \lambda p_{t}\left[\sigma_{3}, \Psi\right]=V\Psi.
\end{aligned}\right.\label{Lax pair 6}
\end{equation}

Furthermore, we define two eigenfunctions $\Psi_{\pm}$ of (\ref{Lax pair 6}) as the solutions of the following two Volterra integral equation in the $(x, t)$ plane
\begin{equation}
\Psi_{\pm}(x, t, \lambda)=I+\int_{\left(x^{*}, \tau^{*}\right)}^{(x, t)} e^{-i \lambda\left(p(x, t, \lambda)-p\left(x^{\prime}, \tau, \lambda\right)\right) \hat{\sigma}_{3}}\left(U\left(x^{\prime}, \tau\right) \Psi_{\pm}\left(x^{\prime}, \tau, \lambda\right) d x^{\prime}+V\left(x^{\prime}, \tau, \lambda\right) \Psi_{\pm}\left(x^{\prime}, \tau, \lambda\right)\right) d \tau.\label{VE}
\end{equation}
Noting that one-form $U\left(x^{\prime}, \tau\right) \Psi_{\pm}\left(x^{\prime}, \tau, \lambda\right) d x^{\prime}+V\left(x^{\prime}, \tau, \lambda\right) \Psi_{\pm}\left(x^{\prime}, \tau, \lambda\right)$ is independent of the path of integration, we choose the particular initial points of integration to be parallel to the $x-axis$ leads to the integral equations for $\Psi_{+}$ and $\Psi_{-}$:
\begin{equation}
\begin{aligned}
&\Psi_{+}(x, t, \lambda)=I-\int_{x}^{\infty} e^{i \lambda \int_{x}^{x^{\prime}} u(\xi, t) d \xi \hat{\sigma}_{3}} U\left(x^{\prime}, t\right) \Psi_{+}\left(x^{\prime}, t, \lambda\right) dx^{\prime}, \\
&\Psi_{-}(x, t, \lambda)=I+\int_{-\infty}^{x} e^{i \lambda \int_{x}^{x^{\prime}} u(\xi, t) d \xi \hat{\sigma}_{3}} U\left(x^{\prime}, t\right) \Psi_{-}\left(x^{\prime}, t, \lambda\right) d x^{\prime}.
\end{aligned}
\end{equation}
Denote $\Psi_{\pm} = (\Psi_{\pm}^{(1)}, \Psi_{\pm}^{(2)}),$ we find that 
$$\Psi_{+}^{(2)}(x, t, \lambda), \Psi_{-}^{(1)}(x, t, \lambda) \text{ are bounded and analytic for } \lambda \in \mathbb{C}_{+},$$
$$\Psi_{+}^{(1)}(x, t, \lambda), \Psi_{-}^{(2)}(x, t, \lambda) \text{ are bounded and analytic for } \lambda \in \mathbb{C}_{-}.$$
\noindent
\textbf{(3) The scattering matrix $s(\lambda)$}
\hspace*{\parindent}

Because the eigenfunctions $\Psi_{\pm}$ are both the solutions of (\ref{VE}), define the spectral function $s(\lambda)$ by
\begin{equation}
\Psi_{-}(x, t, \lambda)=\Psi_{+}(x, t, \lambda) e^{-i \lambda p(x, t, \lambda) \hat{\sigma}_{3}} s(\lambda). \label{s}
\end{equation}
From (\ref{Lax pair 6}), we can deduce that
\begin{equation}
det\left(\Psi_{\pm}(x, t, \lambda)\right)=1. \label{det}
\end{equation}
According the definition of $U(x,t)$, it is obvious that $U(x, t)=\sigma_{1} U(x, t) \sigma_{1},$ we get the $\Psi_{\pm}(x, t, \lambda)$ satisfy the symmetry properties
\begin{equation}
\left\{\begin{aligned}
&\Psi_{\pm 11}(x, t, \lambda)=\overline{\Psi_{\pm 22}(x, t, \bar{\lambda})}, \quad \Psi_{\pm 21}(x, t, \lambda)=\overline{\Psi_{\pm 12}(x, t, \bar{\lambda})}, \\
&\Psi_{\pm 11}(x, t,-\lambda)=\Psi_{\pm 22}(x, t, \lambda), \quad \Psi_{\pm 12}(x, t,-\lambda)=\Psi_{\pm 21}(x, t, \lambda).
\end{aligned}\right.\label{symm}
\end{equation}
Then the matrix $s(\lambda)$ has the form
\begin{equation}
s(\lambda)=\begin{pmatrix}
a(\lambda) & \overline{b(\bar{\lambda})} \\
b(\lambda) & \overline{a(\bar{\lambda})}
\end{pmatrix}, \label{sm}
\end{equation}
where $a(-\bar{\lambda})=\overline{a(\lambda)}$ and $b(-\bar{\lambda})=\overline{b(\lambda)}$, from (\ref{det}) we get $det(s(\lambda))=1$, combining (\ref{s}) and (\ref{det}) gives 
\begin{equation}
a(\lambda) = det(\Psi_{-}^{(1)},\Psi_{+}^{(2)}),\label{adet}  
\end{equation}then we can get $a(\lambda)$ and $b(\lambda)$ have the following properties:
\begin{itemize}
\item $a(\lambda)$ is analytic in $\mathbb{C}_{+}$ and continuous for $\lambda \in \overline{\mathbb{C}}_{+}$, $b(\lambda)$ is continuous for $\lambda \in R. $
\item $a(\lambda)=1+O\left(\frac{1}{\lambda}\right), \lambda \rightarrow \infty$; $a(\lambda)=1+O(\lambda), \lambda \rightarrow 0,$ for $\lambda \in \mathbb{C}_{+}$.
\item $b(\lambda)=O\left(\frac{1}{\lambda}\right), \lambda \rightarrow \infty$; $ b(\lambda)=O(\lambda), \lambda \rightarrow 0,$ for $\lambda \in R.$
\end{itemize}
We introduce the reflection coefficient 
$$r(\lambda)=\frac{b(\lambda)}{a(\lambda)}$$
with symmetry $r(-\bar{\lambda})=\overline{r(\lambda)}$. And for $a(\lambda) \neq 0$ for all $\lambda$ such that $\operatorname{Im} \lambda \geq 0$ \cite{xiao2019}, we assume our initial data satisfy the following assumption.
\begin{property}
The initial data $ q_{0}(x)\in H^{1,1}(\mathbb{R})$ and it generates generic scattering data which satisfy that
\begin{itemize}
\item $a(\lambda)$ only has $N$ simple zeros, denote $\Lambda = \{\lambda_n,a(\lambda_n)=0\text{ and } \lambda \in \mathbb{C}_{+}\}_{n=1}^{N};$
\item $a(\lambda)$ and $r(\lambda)$ belong $H^{1,1}(\mathbb{R})$.
\end{itemize}\label{assump}
\end{property}
\noindent
\textbf{(4) The relation between $\Psi_{\pm}$ and $\Psi_{\pm}^{0}$}
\hspace*{\parindent}

In the following analysis, we formulate a RH problem by defining the matrix function $M(x, t, z)$ with eigenfunctions $\Psi_{\pm}$, while the reconstruction formula between the solution $q(x, t)$ and the RH
problem can be found from the asymptotic of $\Psi_{\pm}$ as $\lambda \rightarrow 0$. So we need to calculate the
relation between $\Psi_{\pm}$ and $\Psi_{\pm}^{0}$.

The eigenfunctions $\Psi_{\pm}$ and $\Psi_{\pm}^{0}$ can be related to each other as 
\begin{equation}
\Psi_{\pm}=G(x, t) \Psi_{\pm}^{0} e^{-i\left(\lambda x+4 \lambda^{3} t\right) \sigma_{3}} C_{\pm}(\lambda) e^{i \lambda p(x, t, \lambda) \sigma_{3}}, \label{relation}
\end{equation}
where $ C_{\pm}(\lambda)$ is independent of $x$ and $t$,
\begin{equation}
G(x, t)=\mathrm{PDP}^{-1}=\frac{1}{2}\begin{pmatrix}
\sqrt{u}+\frac{1}{\sqrt{u}} & \sqrt{u}-\frac{1}{\sqrt{u}} \\
\sqrt{u}-\frac{1}{\sqrt{u}} & \sqrt{u}+\frac{1}{\sqrt{u}}
\end{pmatrix}.
\end{equation}
Considering (\ref{relation}) when $x \rightarrow \pm \infty$, we have $C_{+}(\lambda) = I$ and $C_{-}(\lambda) = e^{ic\lambda \sigma_3}$, where\\
$c=\int_{-\infty}^{\infty}(u(\xi, t)-1) d \xi$ is a quantity conserved under the dynamics governed by (\ref{HDe1}).

Note that (\ref{relation}) can be written as
\begin{equation}
\begin{aligned}
&\Psi_{+}(x, t, \lambda)=G(x, t) \Psi_{+}^0(x, t, \lambda) e^{-i \lambda F(x, t) \sigma_{3}},\\
&\Psi_{-}(x, t, \lambda)=G(x, t) \Psi_{-}^0(x, t, \lambda) e^{i \lambda \int_{-\infty}^{x}(\rho(\xi, t)-1) d \xi \sigma_{3}},
\end{aligned}\label{relation2}
\end{equation}
where $F(x, t)=\int_{x}^{\infty}(u(\xi, t)-1) d \xi$, then we have (\ref{relation}) and (\ref{relation2}) imply that as $\lambda \rightarrow 0$, coefficients of the expansions for $\Psi_{\pm}(x, t, \lambda)$ can be expressed by $q(x, t)$. In addition, according to (\ref{expan1}), (\ref{adet}) and (\ref{relation2}), we find 
\begin{equation}
a(\lambda)=1+i \lambda c+(i \lambda)^{2} \frac{c^{2}}{2}+O\left(\lambda^{3}\right), \quad \text { as } \lambda \rightarrow 0.\label{expan2}
\end{equation}
\section{A Riemann-Hilbert problem}
\hspace*{\parindent}
Suppose that $\Lambda = \{\lambda_n, n = 1,\cdots,N\}$ are simple zeros of $a(\lambda)$, which implies that $\overline{\lambda}_n \in \mathbb{C}_{-}$ are simple zeros of $\overline{a(\bar{\lambda})}$, denote $\bar{\Lambda} = \{\overline{\lambda}_n, n = 1,\cdots,N\}$.

Define a sectionally analytic matrix $M(x, t, \lambda)$
\begin{equation}
M(x, t, \lambda)=\left\{\begin{aligned}
\left(\frac{\Psi_{-}^{(1)}}{a(\lambda)}, \Psi_{+}^{(2)}\right),\quad \lambda \in \mathbb{C}^{+}, \\
\left(\Psi_{+}^{(1)}, \frac{\Psi_{-}^{(2)}}{\overline{a(\bar{\lambda})}}\right),\quad \lambda \in \mathbb{C}^{-}.
\end{aligned}
\right.\label{M1}
\end{equation}
we can get the symmetry of $M(x, t, \lambda)$ by (\ref{symm})
\begin{equation}
\overline{M(x, t, \bar{\lambda})}=M(x, t,-\lambda)=\sigma_{1} M(x, t, \lambda) \sigma_{1}. \label{symm2}
\end{equation}
The function $M$ defined by (\ref{M1}) satisfies the following Riemann-Hilbert problem.\\[6pt]
\textbf{RHP 1}\quad Find a matrix-valued function $M(x, t, \lambda)$ with the following properties
\begin{enumerate}
    \item Analyticity: $M(x, t, \lambda)$ is meromorphic in $\mathbb{C} \backslash \mathbb{R}.$
    \item Jump condition: The boundary values $M_{\pm}$ of $M$ as $\lambda \rightarrow R$ from $\mathbb{C}_{\pm}$ are related by
    \begin{equation}
    M_{+}(x, t, \lambda)=M_{-}(x, t, \lambda) J(x, t, \lambda), \quad \lambda \in R  
    \end{equation}
    where \begin{equation}J(x, t, \lambda)=\begin{pmatrix}
    1-|r(\lambda)|^{2} & -\overline{r(\bar{\lambda})} e^{-2 i \lambda p(x, t, \lambda)} \\[6pt] 
    r(\lambda) e^{2 i \lambda p(x, t, \lambda)} & 1
    \end{pmatrix}. \label{jump1}
     \end{equation}
    \item Normalization: As $\lambda \rightarrow \infty,\quad M(x, t, \lambda) \rightarrow I.$
    \item Residues: At $\lambda=\lambda_{n}$ and $\lambda=\overline{\lambda}_{n}$, $M(x, t, \lambda)$ has simple poles and the residues satisfy the conditions 
    \begin{equation}
\begin{aligned}
&\mathop{\operatorname{Res}}\limits_{\lambda=\lambda_{n}} M(x, t, \lambda)=\lim\limits _{\lambda \rightarrow \lambda_{n}} M(x, t, \lambda){\begin{pmatrix}
0 & 0 \\[6pt]
c_{n}e^{2i\lambda_{n}p(\lambda_n)} & 0
\end{pmatrix}}, \\[6pt]
&\mathop{\operatorname{Res}}\limits_{\lambda=\overline{\lambda}_{n}} M(x, t, \lambda)=\lim\limits_{\lambda \rightarrow \overline{\lambda}_{n}} M(x, t, \lambda)\begin{pmatrix}
0 & \overline{c}_{n} e^{-2 i \overline{\lambda}_{n} p(\overline{\lambda}_{n})} \\[6pt]
0 & 0
\end{pmatrix},
\end{aligned}
\end{equation}
where $c_n = \frac{b(\lambda_n)}{a'(\lambda_n)}.$
\end{enumerate}

Considering the asymptotic behavior for $M(x, t, \lambda)$ when $\lambda \rightarrow 0$, we can solve solution of (\ref{HDe1})-(\ref{chuzhi}) by RHP1. It follows from (\ref{expan1}), (\ref{relation2}) and (\ref{expan2}), for $\lambda \rightarrow 0$ we have
\begin{equation}
\begin{aligned}
M(x, t, \lambda)&=G(x, t)\left\{I-i \lambda\left[\frac{A(x, t)}{2}\left(i \sigma_{2}+\sigma_{3}\right)+F(x, t) \sigma_{3}\right]\right. \\
&\left.+(i \lambda)^{2}\left[A_{1}(x, t) \sigma_{1}+\frac{A(x, t) F(x, t)}{2}\left(I-\sigma_{1}\right)+\frac{F^{2}(x, t)}{2} I\right]+O\left(\lambda^{3}\right)\right\}.
\end{aligned}\label{relation3}
\end{equation}
We introduce a row vector function
\begin{equation}
\left(\mu_{1}, \mu_{2}\right)(x, t, \lambda)=(1, 1) M(x, t, \lambda).
\end{equation}
From (\ref{relation3}), then it follows that as $\lambda \rightarrow 0$,
\begin{equation}
\begin{aligned}
&\mu_{1}(x, t, \lambda)=\sqrt{\rho}\left\{1-i \lambda F(x, t)+(i \lambda)^{2}\left[A_{1}(x, t)+\frac{F^{2}(x, t)}{2}\right]+O\left(\lambda^{3}\right)\right\}, \\
&\mu_{2}(x, t, \lambda)=\sqrt{\rho}\left\{1+i \lambda F(x, t)+(i \lambda)^{2}\left[A_{1}(x, t)+\frac{F^{2}(x, t)}{2}\right]+O\left(\lambda^{3}\right)\right\},
\end{aligned}
\end{equation}
which yield
\begin{equation}
\begin{aligned}
&\frac{\mu_{1}}{\mu_{2}}(x, t, \lambda)=1-2 i \lambda F(x, t)+O\left(\lambda^{2}\right), \\
&q(x, t)=\left(\mu_{1} \mu_{2}\right)^{2}(x, t, 0)-1.
\end{aligned}\label{relation4}
\end{equation}

Equations (\ref{relation4}) show that we can reconstruct the solution of the initial value problem of (\ref{HDe1})-(\ref{chuzhi}) by the matrix-valued function $M(x, t, \lambda)$. However, due to the jump matrix $J$ are determined by $q_{0}(x)$, $e^{-2 i \lambda p(x, t, \lambda)}$, while $p(x, t, \lambda)$ cannot be given by initial data $q_{0}(x)$. Thus, we introduce the new scale
\begin{equation}
y(x, t)=x-\int_{x}^{\infty}(u(\xi, t)-1) d \xi=x-F(x, t),\label{xy}
\end{equation}
which makes the jump matrix $J$ explicitly dependent on the parameters $(y, t)$:
\begin{equation}
J(x, t, \lambda) = \hat{J}(y, t, \lambda)=\begin{pmatrix}
1-|r(\lambda)|^{2} & -\overline{r(\bar{\lambda})} e^{-2 i\left(\lambda y+4 \lambda^{3} t\right)} \\[6pt]
r(\lambda) e^{2 i\left(\lambda y+4 \lambda^{3} t\right)} & 1
\end{pmatrix}, \quad \lambda \in R.\label{jump2}
\end{equation}
By the definition of the new sacle $y(x,t)$, then we can transform this RH problem into a new RH problem parametrized by $(y, t)$. Define
\begin{equation}
\begin{array}{l}
\hat{M}(y, t, \lambda)=M(x(y, t), t, \lambda),\\[6pt]
(\hat{\mu}_{1}, \hat{\mu}_{2})(y, t, \lambda)= (1, 1)\hat{M}(y, t, \lambda).
\end{array}\label{my}
\end{equation}
It follows from (\ref{symm2}) that $\hat{M}(y, t, \lambda)$ has the symmetry property
\begin{equation}
\overline{\hat{M}(y, t, \bar{\lambda})}=\hat{M}(y, t,-\lambda)=\sigma_{1} \hat{M}(y, t, \lambda) \sigma_{1}. \label{symm2}
\end{equation}
We construct the row function $\hat{M}(y, t, \lambda)$ as a solution of the following $2\times2$ matrix Riemann-Hilbert problem.\\[6pt]
\textbf{RHP 2}\quad Find a matrix function $\hat{M}(y, t, \lambda)$ with the following properties
\begin{enumerate}
    \item Analyticity: $\hat{M}(y, t, \lambda)$ is meromorphic in $\mathbb{C} \backslash \mathbb{R}.$
    \item Jump condition: 
    \begin{equation}
    \hat{M}_{+}(y, t, \lambda)=\hat{M}_{-}(y, t, \lambda) \hat{J}(y, t, \lambda), \quad \lambda \in R,
    \end{equation}
    where $\hat{J}(y, t, \lambda)$ is defined by (\ref{jump2}).
    \item Normalization: $\hat{M}(y, t, \lambda) \rightarrow I$, as $\lambda \rightarrow \infty $.
    \item Residues: At $\lambda=\lambda_{n}$ and $\lambda=\overline{\lambda}_{n}$, $\hat{M}(y, t, \lambda)$ has simple poles and the residues satisfy the conditions 
    \begin{equation}
\begin{aligned}
&\mathop{\operatorname{Res}}\limits_{\lambda=\lambda_{n}} \hat{M}(y, t, \lambda)=\lim\limits _{\lambda \rightarrow \lambda_{n}} \hat{M}(y, t, \lambda){\begin{pmatrix}
0 & 0 \\[6pt]
c_{n}e^{2i(\lambda_n y +4\lambda_n^3 t)} & 0
\end{pmatrix}}, \\
&\mathop{\operatorname{Res}}\limits_{\lambda=\overline{\lambda}_{n}} \hat{M}(y, t, \lambda)=\lim\limits_{\lambda \rightarrow \overline{\lambda}_{n}} \hat{M}(y, t, \lambda)\begin{pmatrix}
0 & \overline{c}_{n} e^{-2 i (\bar{\lambda}_n y +4\bar{\lambda}_n^3 t)} \\[6pt]
0 & 0
\end{pmatrix}.
\end{aligned}\label{res1}
\end{equation}
\end{enumerate}
Then, we have
\begin{itemize}
\item The solution $\hat{M}(y, t, \lambda)$ exists and is unique.
\item The solution of the initial value problem of (\ref{HDe1})-(\ref{chuzhi}) can be expressed in parametric form of $(y,t)$ in terms of the solution $\hat{M}(y, t, \lambda)$ of above RHP 2 as follows,
\begin{equation}
q(x,t) = \hat{q}(y, t),
\end{equation}
where 
\begin{equation}
\begin{aligned}
&\hat{q}(y, t)=\left(\hat{\mu}_{1} \hat{\mu}_{2}\right)^{2}(y, t, 0)-1,\\
&x(y, t)=y-\lim\limits_{\lambda \rightarrow 0} \frac{1}{2 i \lambda}\left(\frac{\hat{\mu}_{1}}{\hat{\mu}_{2}}(y, t, \lambda)-1\right),
\end{aligned}\label{relation5}
\end{equation}
and the $\hat{\mu}_1, \hat{\mu}_2$ defined by (\ref{my}).
\end{itemize}
\section{Conjugation}
\hspace*{\parindent}
In this section, we introduce a new transform $\hat{M}(\lambda) \rightarrow M^{(1)}(\lambda)$, from which we make that the $M^{(1)}$ is well behaved as $|t| \rightarrow \infty$ along any characteristic line. Note that the jump matrix $J^{(1)}(y, t, \lambda)$ (see (\ref{jump2})) for the RHP of the the Harry Dym equation involves the exponentials, which is the main factor affecting the jump matrix and the residue conditions. We denote the oscillatory term as
\begin{equation}
e^{-2i(\lambda y +4\lambda^3 t)} = e^{2i\theta(\lambda)},\quad \theta(y, t, \lambda)=\lambda \frac{y}{t}+4 \lambda^{3}.\label{exp}
\end{equation}

Our aim is to study the long time asymptotic behavior of solution $q(x,t)$, the case $\frac{y}{t}>\varepsilon(\varepsilon>0)$ has be discussed by \cite{xiao2019}. We considering the case of $\frac{y}{t}<-\varepsilon(\varepsilon>0)$, which get the stationary phase $\pm\lambda_0$, where $\lambda_{0}=\sqrt{-\frac{y}{12t}} \in \mathbb{R}.$ Then we calculate the new form of $\theta(\lambda)$
\begin{equation}
    \theta(\lambda) = 4\lambda^3-12\lambda\lambda_0^2,
\end{equation}
from which we have
\begin{equation}
    \operatorname{Re}(2i\theta(\lambda)) = 8\operatorname{Im}\lambda((\operatorname{Im}\lambda)^2+3(\lambda_0^2-(\operatorname{Re}\lambda)^2)).\label{sig}
\end{equation}
Then we can get the regions of growth and decay of exponential factor which follows signature table for the function $\operatorname{Re}(2i\theta)$, see Figure \ref{figure1}.
\begin{figure}[htbp]
\centering
\begin{minipage}[t]{0.45\linewidth}
\centering
\begin{tikzpicture}
\node[anchor=south west,inner sep=0] (image) at (0,0)
 {\includegraphics[width=\textwidth]{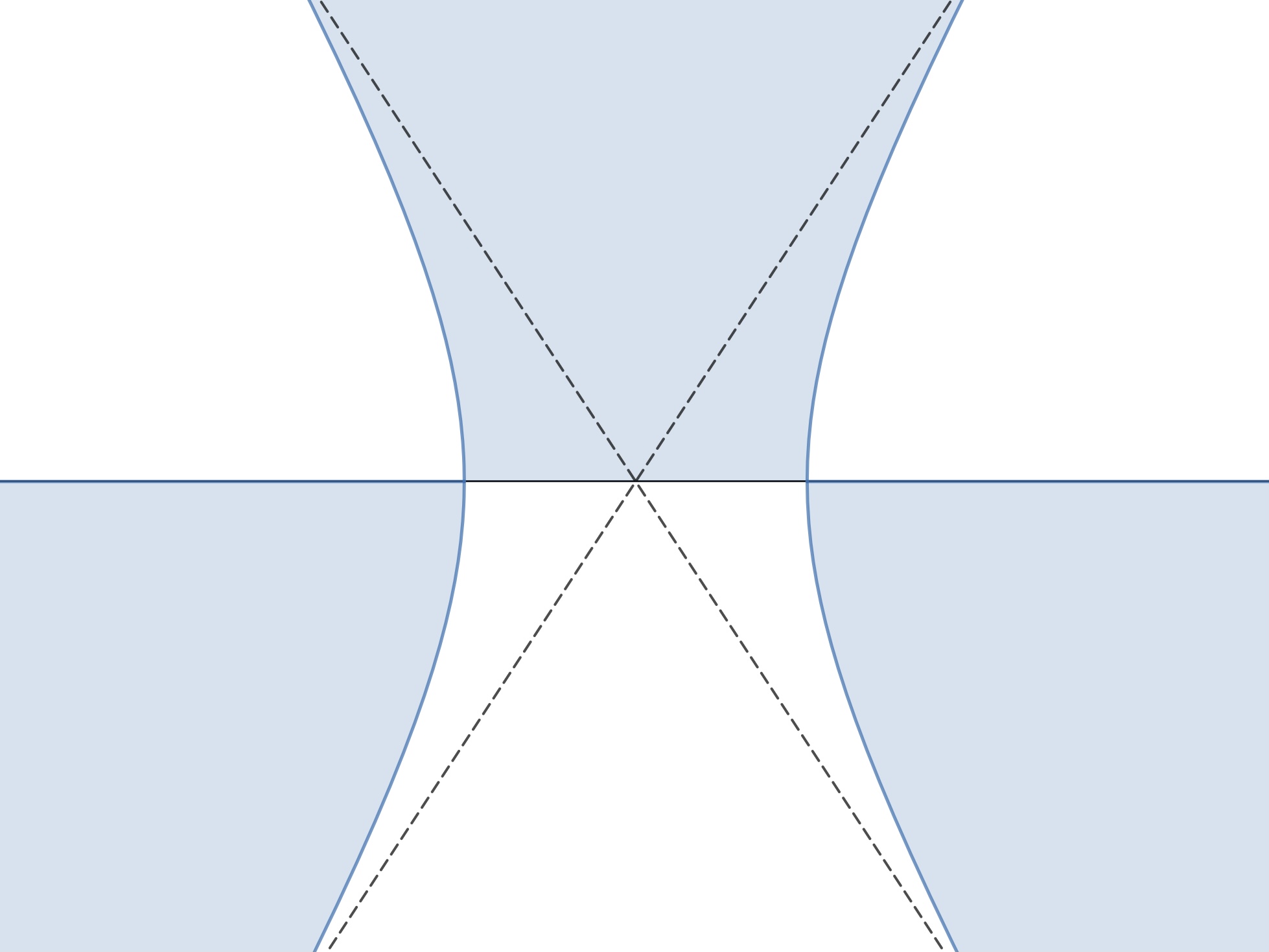}};
    \begin{scope}[x={(image.south east)},y={(image.north west)}]
        \node at (0.5,0.8) {$|e^{2 i t \theta}| \rightarrow \infty$};
        \node at (0.5,0.2) {$|e^{2 i t \theta}| \rightarrow 0$};
        \node at (0.15,0.35) {$|e^{2 i t \theta}| \rightarrow \infty$};
        \node at (0.15,0.65) {$|e^{2 i t \theta}| \rightarrow 0$};
        \node at (0.85,0.35) {$|e^{2 i t \theta}| \rightarrow \infty$};
        \node at (0.85,0.65) {$|e^{2 i t \theta}| \rightarrow 0$};
        \node at (0.5, 0.43) {$O$};
        \node at (0.31, 0.45) {$-\lambda_0$};
        \node at (0.69,0.45) {$\lambda_0$};
        \node at (0.5, 1.05) {$t\rightarrow +\infty$};
    \end{scope}
\end{tikzpicture}
\end{minipage}
\qquad
\begin{minipage}[t]{0.45\linewidth}
\centering
\begin{tikzpicture}
\node[anchor=south west,inner sep=0] (image) at (0,0)
 {\includegraphics[width=\textwidth]{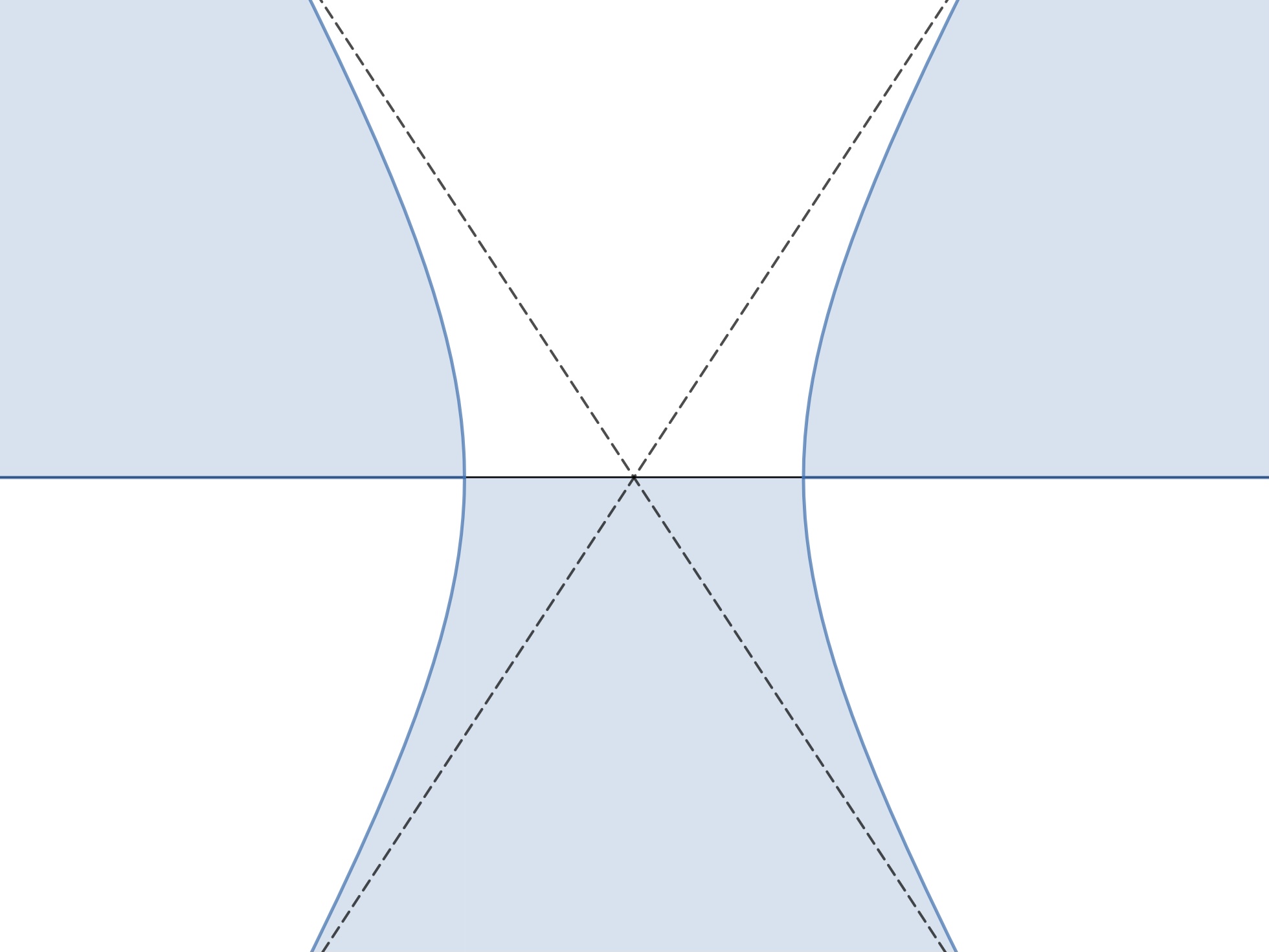}};
    \begin{scope}[x={(image.south east)},y={(image.north west)}]
        \node at (0.5,0.8) {$|e^{2 i t \theta}| \rightarrow 0$};
        \node at (0.5,0.2) {$|e^{2 i t \theta}| \rightarrow \infty$};
        \node at (0.15,0.35) {$|e^{2 i t \theta}| \rightarrow 0$};
        \node at (0.15,0.65) {$|e^{2 i t \theta}| \rightarrow \infty$};
        \node at (0.85,0.35) {$|e^{2 i t \theta}| \rightarrow 0$};
        \node at (0.85,0.65) {$|e^{2 i t \theta}| \rightarrow \infty$};
        \node at (0.5, 0.43) {$O$};
        \node at (0.31, 0.45) {$-\lambda_0$};
        \node at (0.69,0.45) {$\lambda_0$};
        \node at (0.5, 1.05) {$t\rightarrow -\infty$};
    \end{scope}
\end{tikzpicture}
\end{minipage}
\caption{The regions of growth and decay of $|e^{2 i t \theta}|$ for large $|t|$.}
\label{figure1}
\end{figure}
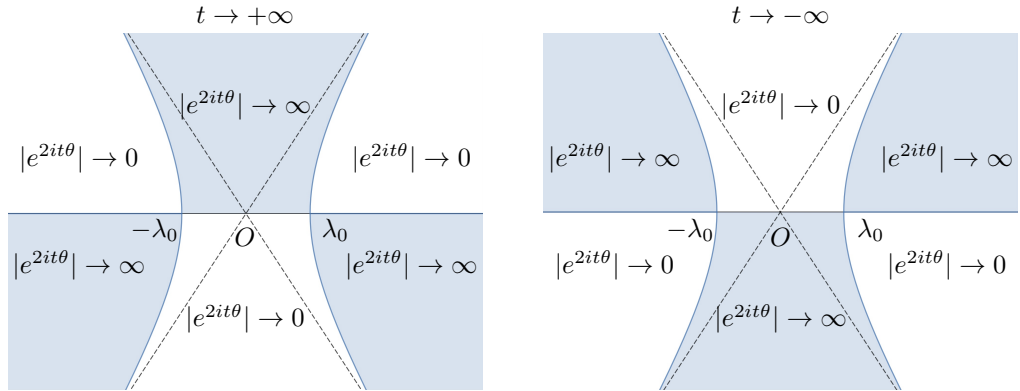

In our analysis, we mainly discuss the case of $t\rightarrow +\infty$, and the case $t\rightarrow -\infty$ can be analyzed in a similarly way. The first step is a conjugation to well-condition the problem for the large-time analysis, we introduce the partition as follow:
$$\Delta_{\lambda_{0}}^{+}=\{n \in \{1,\cdots, N\}||\lambda_n|>\lambda_0\},\quad \Delta_{\lambda_{0}}^{-}=\{n \in \{1,\cdots, N\}||\lambda_n|<\lambda_0\},$$
the intervals on the real-axis are divided as
$$I_{+}=(-\infty,-\lambda_{0}) \cup(\lambda_{0},+\infty), \quad I_{-}=\left[-\lambda_{0}, \lambda_{0}\right].$$

In order to arrive at a problem which is well normalized, we define the following function
\begin{equation}
T(\lambda)=T(\lambda, \lambda_{0})=\prod_{n \in \Delta_{\lambda_{0}}^{-}} \frac{\lambda-\bar{\lambda}_{n}}{\lambda-\lambda_{n}} \delta(\lambda),\label{T}
\end{equation}
where $\delta(\lambda)$ has been solved in \cite{xiao2019}
\begin{equation}
\delta(\lambda)=\left(\frac{\lambda-\lambda_{0}}{\lambda+\lambda_{0}}\right)^{i \nu} \exp{\left(i\int_{-\lambda_{0}}^{\lambda_{0}} \frac{\eta(s)}{s-\lambda}ds\right)},
\end{equation}
\begin{equation}
\nu=-\frac{1}{2 \pi} \log \left(1-\left|r\left(\lambda_{0}\right)\right|^{2}\right)>0, \quad \eta(s)=-\frac{1}{2 \pi}\log \left(\frac{1-|r(s)|^{2}}{1-\left|r\left(\lambda_{0}\right)\right|^{2}}\right).
\end{equation}
\begin{proposition}
The function $T(\lambda)$ defined by (\ref{T}) satisfies following properties:
\begin{enumerate}
    \item $T$ is meromorphic in $\mathbb{C} \backslash I_{-},$ for each $n\in \Delta_{z_{0}}^{-}$, $T(\lambda)$ has a simple pole at $\lambda_n$ and a simple zero at $\bar{\lambda}_n$.  
    \item For $\lambda \in \mathbb{C} \backslash I_{-},$ $T(\lambda)\overline{T(\bar{\lambda})}=1.$
    \item The boundary values $T_{\pm}$ satisfy
    \begin{equation}
    \begin{array}{l}
        T_{+}(\lambda)=T_{-}(\lambda),\quad \lambda \in I_{+},\\[6pt]
         T_{+}(\lambda)=T_{-}(\lambda)\left(1-|r(z)|^{2}\right), \quad \lambda \in I_{-}.
    \end{array}
    \end{equation}
    \item $T(\lambda)$ is continuous at $\lambda=0$ and it admits
    \begin{equation}
T(\lambda)=T(0)\left(1+\lambda T_{1}\right)+\mathcal{O}\left(\lambda^{2}\right),
\end{equation}
where 
\begin{equation}
T_{1}=2 \sum_{n \in \Delta_{\lambda_0}^-} \frac{\operatorname{Im}\left(\lambda_{n}\right)}{\lambda_{n}}-\int_{-\lambda_{0}}^{\lambda_{0}} \frac{\eta(s)}{s^{2}} d s.
\end{equation}
    \item As $|\lambda|\rightarrow \infty$ with $|arg(\lambda)|\leq c<\pi$,
    \begin{equation}
T(\lambda)=1+\frac{i}{\lambda}\left[2 \sum_{n \in \Delta_{\lambda_0}^{-}} \operatorname{Im}\left(\lambda_{n}\right)-\int_{-\lambda_{0}}^{\lambda_{0}}\eta(s)ds-2\nu\lambda_0\right]+\mathcal{O}\left(\lambda^{-2}\right).\label{expan3}
\end{equation}
\item  As $\lambda \rightarrow \pm\lambda_{0}$ along any ray $\lambda = \pm\lambda_{0}+e^{i \omega} \mathbb{R}{+}$ with $|\omega| \leq c<\pi$,
\begin{equation}
\left|T\left(\lambda, \lambda_{0}\right)-T_{0}\left(\pm \lambda_{0}\right)\left(\lambda \mp \lambda_{0}\right)^{i \eta\left(\pm \lambda_{0}\right)}\right| \leq C\left|\lambda \mp \lambda_{0}\right|^{1 / 2},
\end{equation}
where
\begin{equation}
T_{0}\left(\pm \lambda_{0}\right)=T\left(\pm \lambda_{0}, \lambda_{0}\right)=\prod_{n \in \Delta_{\lambda_{0}}^{-}} \frac{\pm \lambda_{0}-\bar{\lambda}_{n}}{\pm \lambda_{0}-\lambda_{n}} e^{i \beta^{\pm}\left(\lambda_{0}, \pm \lambda_{0}\right)},
\end{equation}
\begin{equation}
\beta^{\pm}\left(\lambda, \lambda_{0}\right)=-\eta\left(\pm \lambda_{0}\right) \log \left(\left(\lambda \mp \lambda_{0}+1\right)\right)+\int_{I_{-}} \frac{\eta(s)-\chi_{\pm}(s)\eta\left(\pm \lambda_{0}\right)}{s-\lambda}ds.
\end{equation}
Here $\chi_{\pm}(s)$ are the characteristic functions of the interval $s \in\left(\lambda_{0}-1, \lambda_{0}\right)$ and
$s \in\left(-\lambda_{0},-\lambda_{0}+1\right)$ respectively.
\end{enumerate}
\end{proposition}
\begin{proof}
The proof of above properties can be proved by similar calculation in \cite{yang2019}.
\end{proof}
Now, we introduce the transformation 
\begin{equation}
M^{(1)}=\hat{M} T^{-\sigma_{3}},\label{trans1}
\end{equation}
it is obvious that $M^{(1)}$ keeps same symmetry with $\hat{M}$
$$\overline{M^{(1)}(y, t, \bar{\lambda})}=M^{(1)}(y, t,-\lambda)=\sigma_{1} M^{(1)}(y, t, \lambda) \sigma_{1}.$$
Using the factorization of $\hat{J}$ in this form 
$$\hat{J}(y, t, \lambda)=\left\{\begin{aligned}
&\begin{pmatrix}
1 & -\overline{r(\bar{\lambda})} e^{-2 i t \theta} \\
0 & 1
\end{pmatrix}
\begin{pmatrix}
1 & 0 \\
r(\lambda) e^{2 i t \theta} & 1
\end{pmatrix},|\lambda|>\lambda_{0}, \\
&\begin{pmatrix}
1 & 0 \\
\frac{r(\lambda)}{1-|r(\lambda)|^{2}} e^{2 i t \theta} & 1
\end{pmatrix}
\begin{pmatrix}
1-|r(\lambda)|^{2} & 0 \\
0 & \frac{1}{1-|r(\lambda)|^{2}}
\end{pmatrix}
\begin{pmatrix}
1 & -\frac{\overline{r(\bar{\lambda})}}{1-|r(\lambda)|^{2}} e^{-2 i t \theta} \\
0 & 1
\end{pmatrix},\lambda|<\lambda_{0}.
\end{aligned}\right.$$
According to $J^{(1)}(y, t, \lambda)=(T_{-})^{\sigma_3}\hat{J}(y, t, \lambda)(T_{+})^{-\sigma_3}$, we can calculate the jump matrix $J^{(1)}$ for $M^{(1)}$
\begin{equation}
J^{(1)}(y, t, \lambda)=\left\{\begin{array}{ll}
\begin{pmatrix}
1 & -\overline{r(\bar{\lambda})}T^{2} e^{-2 i t \theta(\lambda)} \\[6pt]
0 & 1
\end{pmatrix}
\begin{pmatrix}
1 & 0 \\[6pt]
r(\lambda) T^{-2} e^{2 i t \theta(\lambda)} & 1
\end{pmatrix}
,\quad|\lambda|>\lambda_{0}, \\[12pt]
\begin{pmatrix}
1 & 0 \\[6pt]
\frac{r(\lambda)}{1-|r(\lambda)|^{2}} T_{-}^{-2} e^{2 i t \theta(\lambda)} & 1
\end{pmatrix}
\begin{pmatrix}
1 & -\frac{\overline{r(\bar{\lambda})}}{1-|r(\lambda)|^{2}} T_{+}^{2} e^{-2 i t \theta(\lambda)} \\[6pt]
0 & 1
\end{pmatrix}
,\quad|\lambda|<\lambda_{0}.
\end{array}\right.\label{jump3}
\end{equation}
Correspondingly, when $\mu^{(1)}(y, t, \lambda)=(1\quad1) M^{(1)}(y, t, \lambda)$, we have $\mu^{(1)}(y, t, \lambda)=\hat{\mu}(y, t, \lambda) T^{-\sigma_{3}}$.
Thus the transformation (\ref{trans1}) can be expressed as
$$
\left(\hat{\mu}_{1}, \hat{\mu}_{2}\right)(y, t, \lambda) T^{-\sigma_{3}}(\lambda)=(1, 1) M^{(1)}(y, t, \lambda).
$$
In order to solve (\ref{relation5}), we calculate that
\begin{equation}
\left(\hat{\mu}_{1} \hat{\mu}_{2}\right)(y, t, 0) T^{-2}(0)=\left(M_{11}^{(1)}+M_{21}^{(1)}\right)\left(M_{12}^{(1)}+M_{22}^{(1)}\right)(y, t, 0),\label{relation6}
\end{equation}
where $M_{11}^{(1)}(y, t, 0), M_{21}^{(1)}(y, t, 0), M_{12}^{(1)}(y, t, 0), M_{22}^{(1)}(y, t, 0)$ represent the 11,21,12,22 element of $M^{(1)}(y, t, 0)$ respectively.
The function defined by (\ref{trans1}) satisfies the following Riemann-Hilbert problem.\\[6pt]
\textbf{RHP 3}\quad Find a matrix function $M^{(1)}(y, t, \lambda)$ with the following properties
\begin{enumerate}
    \item Analyticity: $M^{(1)}(y, t, \lambda)$ is meromorphic in $\mathbb{C} \backslash \mathbb{R}.$
    \item Jump condition: 
    \begin{equation}
    M^{(1)}_{+}(y, t, \lambda)=M^{(1)}_{-}(y, t, \lambda) J^{(1)}(y, t, \lambda), \quad \lambda \in R,
    \end{equation}
    where $J^{(1)}(y, t, \lambda)$ is defined by (\ref{jump3}).
    \item Normalization: $M^{(1)}(y, t, \lambda) \rightarrow I$, as $\lambda \rightarrow \infty $.
    \item Residues: At $\lambda=\lambda_{n}$ and $\lambda=\overline{\lambda}_{n}$, $M^{(1)}(y, t, \lambda)$ has simple poles and the residues satisfy the conditions\\
    For $n\in \Delta_{\lambda_{0}}^{+}$,
    \begin{equation}
\begin{aligned}
&\mathop{\operatorname{Res}}\limits_{\lambda=\lambda_{n}} M^{(1)}(y, t, \lambda)=\lim\limits _{\lambda \rightarrow \lambda_{n}} M^{(1)}(y, t, \lambda){\begin{pmatrix}
0 & 0 \\[6pt]
c_{n}T^{-2}(\lambda_n)e^{2i\theta_n t} & 0
\end{pmatrix}}, \\
&\mathop{\operatorname{Res}}\limits_{\lambda=\overline{\lambda}_{n}} M^{(1)}(y, t, \lambda)=\lim\limits_{\lambda \rightarrow \overline{\lambda}_{n}} M^{(1)}(y, t, \lambda)\begin{pmatrix}
0 & \overline{c}_{n}T^2(\bar{\lambda}_n) e^{-2 i\bar{\theta}_n t} \\[6pt]
0 & 0
\end{pmatrix}.
\end{aligned}\label{res3}
\end{equation}
For $n\in \Delta_{\lambda_{0}}^{-}$,
\begin{equation}
\begin{array}{l}
\mathop{\operatorname{Res}}\limits_{\lambda=\lambda_{n}} M^{(1)}(y, t, \lambda)=\lim\limits _{\lambda \rightarrow \lambda_{n}} M^{(1)}(y, t, \lambda){\begin{pmatrix}
0 & c_{n}^{-1}(1 / T)^{\prime}\left(\lambda_{n}\right)^{-2}e^{-2i\theta_n t} \\[6pt]
0 & 0
\end{pmatrix}}, \\[15pt]
\mathop{\operatorname{Res}}\limits_{\lambda=\overline{\lambda}_{n}} M^{(1)}(y, t, \lambda)=\lim\limits_{\lambda \rightarrow \overline{\lambda}_{n}} M^{(1)}(y, t, \lambda)\begin{pmatrix}
0 & 0 \\[6pt]
(\overline{c}_{n})^{-1}T^{\prime}\left(\bar{\lambda}_{n}\right)^{-2}e^{2 i \bar{\theta}_n t} & 0
\end{pmatrix},
\end{array}\label{res4}
\end{equation}
where $\theta_n=\theta(y,t,\lambda_n).$
\end{enumerate}
\begin{proof}
Note that the properties $1-3$ of $M^{(1)}(y, t, \lambda)$ follow (\ref{trans1}) and RHP 2, we only prove the residue conditions. Denote $M^{(1)}=(M^{(1)}_1,M^{(1)}_2)$, then (\ref{trans1}) becomes
\begin{equation}
  M^{(1)}=(\hat{M}_1T^{-1}, \hat{M}_2T).\label{res2}  
\end{equation}

For $n\in \Delta_{\lambda_{0}}^{+}$, $T$ and $T^{-1}$ is analytic at each $\lambda_n$ and $\bar{\lambda}_n$, combine (\ref{res1}) and (\ref{res2}), we can get (\ref{res3}) immediately.

For $n\in \Delta_{\lambda_{0}}^{-}$, $\lambda_n$ is a pole for $T$ and $T^{-1}(\lambda_n)=0$, we have
\begin{equation}
\begin{array}{l}
\mathop{\operatorname{Res}}\limits_{\lambda=\lambda_{n}} M_{1}^{(1)}=0,\\[8pt]
\mathop{\operatorname{Res}}\limits_{\lambda=\lambda_{n}} M_{2}^{(1)}=\hat{M}_{2}\left(\lambda_{n}\right)\left[(1/ T)^{\prime}\left(\lambda_{n}\right)\right]^{-1}.
\end{array}\label{cal1}
\end{equation}
And it follows (\ref{trans1}) that
\begin{equation}
M_{1}^{(1)}\left(\lambda_{n}\right)=\lim _{\lambda \rightarrow \lambda_{n}} \hat{M}_{1} T=\mathop{\operatorname{Res}}\limits_{\lambda=\lambda_{n}} \hat{M}_{1}(1 / T)^{\prime}\left(\lambda_{n}\right).\label{cal2}
\end{equation}
According to (\ref{res1}),(\ref{cal1}) and (\ref{cal2}), the residues of $M^{(1)}(y, t, \lambda)$ at $\lambda_n$ can be expressed as (\ref{res4}). And the residues of $M^{(1)}(y, t, \lambda)$ at $\bar{\lambda}_n$ can be calculated in a similarly way.
\end{proof}
\section{A mixed $\bar{\partial}$ -RH problem}
\hspace*{\parindent}
In this section, we introduce a transformation $M^{(1)}(\lambda) \rightarrow M^{(2)}(\lambda)$ which make continuous extensions off the real axis to the jump matrix. Define new contours 
\begin{equation}
\begin{aligned}
&\Sigma_{k}=\left\{\lambda_{0}+h_{k} \lambda_{0} e^{\frac{(2 k-1) \pi i}{4}}, h_{k} \in(0,+\infty)\right\}, \quad k=1,4,\\
&\Sigma_{k}=\left\{\lambda_{0}+h_{k} \lambda_{0} e^{\frac{(2 k-1) \pi i}{4}}, h_{k} \in(0, \sqrt{2})\right\}, \quad k=2,3,\\
&\Sigma_{k}=\left\{-\lambda_{0}+h_{k} \lambda_{0} e^{\frac{(2 k-1) \pi i}{4}}, h_{k} \in(0,+\infty)\right\}, \quad k=6,7,\\
&\Sigma_{k}=\left\{-\lambda_{0}+h_{k} \lambda_{0} e^{\frac{(2 k-1) \pi i}{4}}, h_{k} \in(0, \sqrt{2})\right\}, \quad k=5,8,\\
&\Sigma^{(2)}=\bigcup_{k=1}^{8} \Sigma_{k}.
\end{aligned}
\end{equation}
Thus the complex plane $\mathbb{C}$ is divided into eight open sectors, we apply $\bar{\partial}$ steepest descent method to continuously extend the scattering data in the jump matrix to eight regions and there are no more jumps on the real axis. Denote these regions in counterclockwise order by $\Omega_{k}, k=1, \ldots, 8$ respectively, see Figure \ref{f3}. 
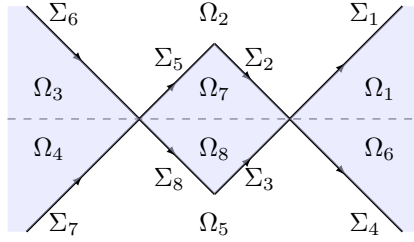
\begin{figure}[H]
 \centering
  \begin{tikzpicture}[node distance=2cm]
  \draw[dashed] (-2.75, 0) -- (2.75, 0);
  \draw[thick ](0,1)--(1,0);
  \draw[-latex](0,1)--(0.5, 0.5);
  \draw[thick ](1,0)--(2.5,-1.5);
  \draw[-latex](1,0)--(1.75,-0.75);
  \draw[thick ](-1,0)--(0,1);
  \draw[-latex](-1,0)--(-0.5, 0.5);
  \draw[thick ](-1,0)--(0,-1);
  \draw[-latex](-1,0)--(-0.5, -0.5);
  \draw[thick ](0,-1)--(1,0);
  \draw[-latex](0,-1)--(0.5, -0.5);
  \draw[thick ](1,0)--(2.5,1.5);
  \draw[-latex](1,0)--(1.75,0.75);
  \draw[thick ](-1,0)--(-2.5,1.5);
  \draw[-latex](-2.5,1.5)--(-1.75,0.75);
    \draw[thick ](-1,0)--(-2.5,-1.5);
  \draw[-latex](-2.5,-1.5)--(-1.75,-0.75);
  \fill[fill=blue!15,fill opacity=0.5] (-1,0)--(0,1)--(1,0)--(0,-1);
  \fill[fill=blue!15,fill opacity=0.5] (1,0)--(2.5,1.5)--(2.75,1.5)--(2.75,-1.5)--(2.5,-1.5);
  \fill[fill=blue!15,fill opacity=0.5] (-1,0)--(-2.5,1.5)--(-2.75,1.5)--(-2.75,-1.5)--(-2.5,-1.5);
  \node at (0, 0.4) {$\Omega_7$};
  \node at (0, -0.4) {$\Omega_8$};
  \node at (2.2, 0.4) {$\Omega_1$};
  \node at (2.2, -0.4) {$\Omega_6$};
  \node at (-2.2, 0.4) {$\Omega_3$};
  \node at (-2.2, -0.4) {$\Omega_4$};
  \node at (0,1.4) {$\Omega_2$};
  \node at (0,-1.4) {$\Omega_5$};
  \node at (-2, 1.4) {$\Sigma_6$};
  \node at (-2, -1.4) {$\Sigma_7$};
  \node at (2, 1.4) {$\Sigma_1$};
  \node at (2, -1.4) {$\Sigma_4$};
  \node at (0.6, 0.8) {$\Sigma_2$};
  \node at (0.6,-0.8) {$\Sigma_3$};
  \node at (-0.6, 0.8) {$\Sigma_5$};
  \node at (-0.6,-0.8) {$\Sigma_8$};
  \end{tikzpicture}
\caption{The contours of $ \Sigma^{(2)} $ and regions.}.
\label{f3}
\end{figure}

Let
\begin{equation}
d =\frac{1}{2} \min _{\alpha \neq \beta \in \Lambda \cup \overline{\Lambda}}|\alpha-\beta|.\label{dist}
\end{equation}
For $\forall \lambda_{n}=x_{n}+i y_{n} \in \Lambda$, since the poles are conjugated and not on the real axis, we have $\bar{\lambda}_{n}=x_{n}+i y_{n} \in \overline{\Lambda}$, then $\operatorname{dist}(\Lambda, \mathbb{R}) \geq d$. Next we introduce the characteristic function near the discrete spectrum
\begin{equation}
X_{\Lambda}(\lambda)=\left\{\begin{array}{ll}
1, & \operatorname{dist}\left(\lambda, \Lambda \cup \overline{\Lambda}\right)<d / 3, \\[6pt]
0, & \operatorname{dist}\left(\lambda, \Lambda \cup \overline{\Lambda}\right)>2 d / 3.
\end{array}\right.
\end{equation}
For our purpose of continuous extension to the scattering data in the jump matrix, we introduce a transformation
\begin{equation}
M^{(2)}(\lambda)=M^{(1)}(\lambda) \mathcal{R}^{(2)}(\lambda),\label{trans2}
\end{equation}
Correspondingly, we define $(\mu^{(2)}_1\quad \mu^{(2)}_2)(y, t, \lambda)=(1\quad1) M^{(2)}(y, t, \lambda)$, and $R^{(2)}$ is chosen to satisfy the following properties:
\begin{itemize}
\item $M^{(2)}(\lambda)$ has no jump on the real axis;
\item The norm of $M^{(2)}(\lambda)$ is well be controlled;
\item The transformation keep the residue conditions unchanged.
\end{itemize}
In order to meet the above properties, we can choose the boundary values of $\mathcal{R}^{(2)}(\lambda)$ by matching the transformed RH problem to a well known model RH problem, then the following proposition holds.
\begin{proposition}
It is possible to define functions $R_{j}: \bar{\Omega}_{j} \rightarrow C, j=1,3,4,6,7,8$ which satisfy the boundary conditions\\
$$
\begin{aligned}
&R_{1}(\lambda)=\left\{\begin{aligned}
&r(\lambda) T(\lambda)^{-2}, & \lambda \in\left(\lambda_{0},+\infty\right), \\
&r\left(\lambda_{0}\right) T_{0}\left(\lambda_{0}\right)^{-2}\left(\lambda-\lambda_{0}\right)^{-2 i \eta\left(\lambda_{0}\right)}\left(1-X_{\Lambda}(\lambda)\right), &\qquad \qquad \qquad\lambda \in \Sigma_{1},
\end{aligned}\right. \\[6pt]
&R_{3}(\lambda)=\left\{\begin{aligned}
&r(\lambda) T(\lambda)^{-2}, & \lambda \in\left(-\infty,-\lambda_{0}\right), \\
&r\left(-\lambda_{0}\right) T_{0}\left(-\lambda_{0}\right)^{-2}\left(\lambda+\lambda_{0}\right)^{-2 i \eta\left(-\lambda_{0}\right)}\left(1-X_{\Lambda}(\lambda)\right), & \lambda \in \Sigma_{6},
\end{aligned}\right.\\[6pt]
&R_{4}(\lambda)=\left\{\begin{aligned}
&\overline{r(\bar{\lambda})} T(\lambda)^{2}, &\qquad \lambda \in\left(-\infty,-\lambda_{0}\right), \\
&\overline{r\left(-\lambda_{0}\right)} T_0\left(-\lambda_{0}\right)^{2}\left(\lambda+\lambda_{0}\right)^{2 i \eta\left(-\lambda_{0}\right)}\left(1-X_{\Lambda}(\lambda)\right), & \lambda \in \Sigma_{7},
\end{aligned}\right.\\[6pt]
&R_{6}(\lambda)=\left\{\begin{aligned}
&\overline{r(\bar{\lambda})} T(\lambda)^{2}, & \qquad \quad  \qquad\lambda \in \left(\lambda_{0},+\infty\right), \\
&\overline{r\left(\lambda_{0}\right)} T_0\left(\lambda_{0}\right)^{2}\left(\lambda-\lambda_{0}\right)^{2 i \eta\left(\lambda_{0}\right)}\left(1-X_{\Lambda}(\lambda)\right), & \lambda \in \Sigma_{4},
\end{aligned}\right.\\[6pt]
&R_{7}(\lambda)=\left\{\begin{aligned}
&\frac{\overline{r(\bar{\lambda})}}{1-|r(\lambda)|^{2}} T_{+}(\lambda)^{2}, & \lambda \in \left(-\lambda_{0}, \lambda_{0}\right), \\
&\frac{\overline{\gamma\left(\lambda_{0}\right)}}{1-\left|r\left(\lambda_{0}\right)\right|^{2}} T_{0}\left(\lambda_{0}\right)^{2}\left(\lambda-\lambda_{0}\right)^{2 i \eta\left(\lambda_{0}\right)}\left(1-X_{\Lambda}(\lambda)\right), & \lambda \in \Sigma_{2}, \\
&\frac{\overline{\gamma\left(-\lambda_{0}\right)}}{1-\left|r\left(-\lambda_{0}\right)\right|^{2}} T_{0}\left(-\lambda_{0}\right)^{2}\left(\lambda+\lambda_{0}\right)^{2 i \eta\left(-\lambda_{0}\right)}\left(1-X_{\Lambda}(\lambda)\right), & \lambda \in \Sigma_{5},
\end{aligned}\right.\\[6pt]
&R_{8}(\lambda)=\left\{\begin{aligned}
&\frac{r(\lambda)}{1-|r(\lambda)|^{2}} T_-(\lambda)^{-2}, & \lambda \in\left(-\lambda_{0}, \lambda_{0}\right), \\
&\frac{r\left(\lambda_{0}\right)}{1-|r(\lambda_0)|^2} T_0\left(\lambda_{0}\right)^{-2}\left(\lambda-\lambda_{0}\right)^{-2 i \eta\left(\lambda_{0}\right)}\left(1-X_{\Lambda}(\lambda)\right), & \lambda \in \Sigma_{3}, \\
&\frac{r\left(-\lambda_{0}\right)}{1-\left|r\left(-\lambda_{0}\right)\right|^{2}} T_{0}\left(-\lambda_{0}\right)^{-2}\left(\lambda+\lambda_{0}\right)^{-2 i \eta\left(-\lambda_{0}\right)}\left(1-X_{\Lambda}(\lambda)\right), & \lambda \in \Sigma_{8}.
\end{aligned}\right.
\end{aligned}
$$
And we have the norm estimation\\
for $j=1,6,7,8,$
\begin{equation}
\begin{array}{l}
\left|R_{j}(\lambda)\right| \lesssim \sin ^{2}\left(\arg \left(\lambda-\lambda_{0}\right)\right)+\langle\operatorname{Re}(\lambda)\rangle^{-1 / 2}, \\[6pt]
\left|\overline{\partial} R_{j}(\lambda)\right| \lesssim\left|\bar{\partial} X_{\Lambda}(\lambda)\right|+\left|s_{j}^{\prime}(\operatorname{Re\lambda})\right|+\left|\lambda-\lambda_{0}\right|^{-1 / 2},
\end{array}\label{est3}
\end{equation}
for $j=3,4,7,8,$
\begin{equation}
\begin{array}{l}
\left|R_{j}(\lambda)\right| \lesssim \sin ^{2}\left(\arg \left(\lambda+\lambda_{0}\right)\right)+\langle\operatorname{Re}(\lambda)\rangle^{-1 / 2}, \\[6pt]
\left|\overline{\partial} R_{j}(\lambda)\right| \lesssim\left|\bar{\partial} X_{\Lambda}(\lambda)\right|+\left|s_{j}^{\prime}(\operatorname{Re\lambda})\right|+\left|\lambda+\lambda_{0}\right|^{-1 / 2},
\end{array}\label{est4}
\end{equation}
where 
$$\langle\cdot\rangle:=\sqrt{1+(\cdot)^{2}},$$
$$ s_1=s_3=r(\lambda),\quad s_4=s_6=\overline{r(\bar{\lambda})},$$
$$s_7=\frac{\overline{r(\bar{\lambda})}}{1-|r(\lambda)|^{2}},\quad s_8=\frac{r(\lambda)}{1-|r(\lambda)|^{2}}.$$
And \begin{equation}
\bar{\partial} R_{j}(\lambda)=0, \quad \text { if } \lambda \in \Omega_{2} \cup \Omega_{5} \text { or } \operatorname{dist}(\lambda, \Lambda \cup \overline{\Lambda})<d / 3.
\end{equation}
\end{proposition}
\begin{proof}
The proof is similar to that in \cite{yang2019}.
\end{proof}
Then we can define $\mathcal{R}^{(2)}$ as 
\begin{equation}
\mathcal{R}^{(2)}(\lambda)=\left\{\begin{aligned}
&\begin{pmatrix}
1 & 0 \\[6pt]
-R_{j}(\lambda) e^{2 i t \theta} & 1
\end{pmatrix}, & \lambda \in \Omega_{j}, j=1,3,\\[6pt]
&\begin{pmatrix}
1 & -R_{j}(\lambda) e^{-2 i t \theta} \\[6pt]
0 & 1
\end{pmatrix},& \lambda \in \Omega_{j}, j=4,6, \\[6pt]
&\begin{pmatrix}
1 & R_{7}(\lambda)e^{-2 i t \theta} \\[6pt]
0 & 1
\end{pmatrix}, & \lambda \in \Omega_{7}, \\[6pt]
&\begin{pmatrix}
1 & 0 \\[6pt]
R_{8}(\lambda) e^{2 i t \theta} & 1
\end{pmatrix},& \lambda \in \Omega_{8}, \\[6pt]
&\quad I,& \lambda \in \Omega_{2} \cup \Omega_{5}.
\end{aligned}\right. 
\end{equation}
According to the above definition of $\mathcal{R}^{(2)}(\lambda)$ and (\ref{trans2}), we get the $M^{(2)}$ which domain is $\mathbb{C} \backslash\left(\Sigma^{(2)}\cup \Lambda \cup \overline{\Lambda}\right)$ satisfy the following $\bar{\partial}-\mathrm{RH}$ problem.\\[6pt]
\textbf{RHP 4}\quad Find a matrix function $M^{(2)}(y, t, \lambda)$ with the following properties
\begin{enumerate}
    \item Analyticity: $M^{(2)}(y, t, \lambda)$ is meromorphic in $\mathbb{C} \backslash \Sigma^{(2)}.$
    \item Jump condition: 
    \begin{equation}
    M^{(2)}_{+}(y, t, \lambda)=M^{(2)}_{-}(y, t, \lambda) J^{(2)}(y, t, \lambda), \quad \lambda \in \Sigma^{(2)}, 
    \end{equation}
    where $J^{(2)}(y, t, \lambda)$ is defined as
    \begin{equation}
J^{(2)}(y, t, \lambda)=\left\{
\begin{array}{ll}
\begin{pmatrix}
1 & 0 \\[6pt]
R_{1}(\lambda) e^{2 i t \theta} & 1
\end{pmatrix},& \lambda \in \Sigma_{1}, \\[15pt]
\begin{pmatrix}
1 & -R_{7}(\lambda) e^{-2i t \theta} \\[6pt]
0 & 1
\end{pmatrix},& \lambda \in \Sigma_{2} \cup \Sigma_{5}, \\[15pt]
\begin{pmatrix}
1 & 0 \\[6pt]
R_{8}(\lambda) e^{2 i t \theta} & 1
\end{pmatrix},& \lambda \in \Sigma_{3} \cup \Sigma_{8},\\[15pt]
\begin{pmatrix}
1 & -R_{6}(\lambda) e^{-2 i t \theta} \\[6pt]
0 & 1
\end{pmatrix}, & \lambda \in \Sigma_{4}, \\[15pt]
\begin{pmatrix}
1 & 0 \\[6pt]
R_{3}(\lambda) e^{2 i t \theta} & 1
\end{pmatrix},& \lambda \in \Sigma_{6},\\[15pt]
\begin{pmatrix}
1 & -R_{4} \left(\lambda\right) e^{2 i t \theta} \\[6pt]
0 & 1
\end{pmatrix},& \lambda \in \Sigma_{7}.
\end{array}\right.\label{jump4}
\end{equation}
    \item Normalization: $M^{(2)}(y, t, \lambda) \rightarrow I$, as $\lambda \rightarrow \infty $.
    \item $\bar{\partial}$-Derivative: For $\lambda \in \mathbb{C} \backslash\left(\Sigma^{(2)}\cup \Lambda \cup \overline{\Lambda}\right)$ we have 
    \begin{equation}
\bar{\partial} M^{(2)}=M^{(2)} \bar{\partial} \mathcal{R}^{(2)},\label{par}
\end{equation}
where
\begin{equation}
\bar{\partial}\mathcal{R}^{(2)}(\lambda)=\left\{\begin{array}{ll}
\begin{pmatrix}
0 & 0 \\[6pt]
-\bar{\partial}R_{j}(\lambda) e^{2 i t \theta} & 0
\end{pmatrix}, & \lambda \in \Omega_{j}, j=1,3,\\[15pt]
\begin{pmatrix}
0 & -\bar{\partial}R_{j}(\lambda) e^{-2 i t \theta} \\[6pt]
0 & 0
\end{pmatrix},& \lambda \in \Omega_{j}, j=4,6, \\[15pt]
\begin{pmatrix}
0 & \bar{\partial}R_{7}(\lambda)e^{-2 i t \theta} \\[6pt]
0 & 0
\end{pmatrix}, & \lambda \in \Omega_{7}, \\[15pt]
\begin{pmatrix}
0 & 0 \\[6pt]
\bar{\partial}R_{8}(\lambda) e^{2 i t \theta} & 0
\end{pmatrix},& \lambda \in \Omega_{8}, \\[15pt]
\quad 0, & \lambda \in \Omega_{2} \cup \Omega_{5}.
\end{array}\right. 
\end{equation}
    \item Residues: At $\lambda=\lambda_{n}$ and $\lambda=\overline{\lambda}_{n}$, $M^{(2)}(y, t, \lambda)$ has simple poles and the residues satisfy the conditions\\
    For $n\in \Delta_{\lambda_{0}}^{+}$,
    \begin{equation}
\begin{array}{l}
\mathop{\operatorname{Res}}\limits_{\lambda=\lambda_{n}} M^{(2)}(y, t, \lambda)=\lim\limits _{\lambda \rightarrow \lambda_{n}} M^{(2)}(y, t, \lambda){\begin{pmatrix}
0 & 0 \\[6pt]
c_{n}T^{-2}(\lambda_n)e^{2i\theta_n t} & 0
\end{pmatrix}}, \\[15pt]
\mathop{\operatorname{Res}}\limits_{\lambda=\overline{\lambda}_{n}} M^{(2)}(y, t, \lambda)=\lim\limits_{\lambda \rightarrow \overline{\lambda}_{n}} M^{(2)}(y, t, \lambda)\begin{pmatrix}
0 & \overline{c}_{n}T^2(\bar{\lambda}_n) e^{-2 i \bar{\theta}_n t} \\[6pt]
0 & 0
\end{pmatrix}.
\end{array}\label{res3}
\end{equation}
For $n\in \Delta_{\lambda_{0}}^{-}$,
\begin{equation}
\begin{array}{l}
\mathop{\operatorname{Res}}\limits_{\lambda=\lambda_{n}} M^{(2)}(y, t, \lambda)=\lim\limits _{\lambda \rightarrow \lambda_{n}} M^{(2)}(y, t, \lambda){\begin{pmatrix}
0 & c_{n}^{-1}(1 / T)^{\prime}\left(\lambda_{n}\right)^{-2}e^{-2i\theta_n t} \\[6pt]
0 & 0
\end{pmatrix}}, \\[15pt]
\mathop{\operatorname{Res}}\limits_{\lambda=\overline{\lambda}_{n}} M^{(2)}(y, t, \lambda)=\lim\limits_{\lambda \rightarrow \overline{\lambda}_{n}} M^{(2)}(y, t, \lambda)\begin{pmatrix}
0 & 0 \\[6pt]
(\overline{c}_{n})^{-1}T^{\prime}\left(\bar{\lambda}_{n}\right)^{-2}e^{2 i \bar{\theta}_n t} & 0
\end{pmatrix}.
\end{array}\label{res4}
\end{equation}
\end{enumerate} 

\section{Decomposition of the mixed $\bar{\partial}$ -RH problem}
\hspace*{\parindent}
In this section, our goal is to decompose this mixed $\bar{\partial}$ -RH problem into two parts according to whether the $\bar{\partial} \mathcal{R}^{(2)}$ is equal to zero and then solve them separately. The decomposition can be expressed as 
\begin{equation}
M^{(2)}(y, t, \lambda)=\left\{\begin{aligned}
&M_{RHP}^{(2)}(y, t, \lambda), & \bar{\partial} \mathcal{R}^{(2)}=0, \\
&M^{(3)}(y, t, \lambda) M_{RHP}^{(2)}(y, t, \lambda), &\bar{\partial} \mathcal{R}^{(2)} \neq 0,
\end{aligned}\right.\label{dec}
\end{equation}
where $M_{RHP}^{(2)}$ indicates the pure RH part in the mixed $\bar{\partial}$ -RH problem, which implies $M_{RHP}^{(2)}$ satisfies the same jump condition and residues. 

For the first step, considering the case of $\bar{\partial} \mathcal{R}^{(2)}=0$, we establish a RH problem for $M_{RHP}^{(2)}$ as follows.\\[6pt]
\textbf{RHP 5}\quad Find a matrix function $M_{RHP}^{(2)}(y, t, \lambda)$ with the following properties
\begin{enumerate}
    \item Analyticity: $M_{RHP}^{(2)}(y, t, \lambda)$ is meromorphic in $\mathbb{C} \backslash \Sigma^{(2)}.$
    \item Jump condition: 
    \begin{equation}
    M_{RHP+}^{(2)}=M_{RHP-}^{(2)} J^{(2)}(y, t, \lambda), \quad \lambda \in \Sigma^{(2)},
    \end{equation}
    where $J^{(2)}(y, t, \lambda)$ is defined as (\ref{jump4}).
    \item Normalization: $M_{RHP}^{(2)}(y, t, \lambda) \rightarrow I$, as $\lambda \rightarrow \infty $.
    \item $\bar{\partial}$-Derivative: For $\lambda \in \mathbb{C},\quad \bar{\partial} \mathcal{R}^{(2)}=0.$
    \item Residues: At $\lambda=\lambda_{n}$ and $\lambda=\overline{\lambda}_{n}$, $M_{RHP}^{(2)}(y, t, \lambda)$ has simple poles and the residues satisfy the conditions\\
    For $n\in \Delta_{\lambda_{0}}^{+}$,
    \begin{equation}
\begin{array}{l}
\mathop{\operatorname{Res}}\limits_{\lambda=\lambda_{n}} M_{RHP}^{(2)}(y, t, \lambda)=\lim\limits _{\lambda \rightarrow \lambda_{n}} M_{RHP}^{(2)}(y, t, \lambda){\begin{pmatrix}
0 & 0 \\[6pt]
c_{n}T^{-2}(\lambda_n)e^{2i\theta_n t} & 0
\end{pmatrix}}, \\[15pt]
\mathop{\operatorname{Res}}\limits_{\lambda=\overline{\lambda}_{n}} M_{RHP}^{(2)}(y, t, \lambda)=\lim\limits_{\lambda \rightarrow \overline{\lambda}_{n}} M_{RHP}^{(2)}(y, t, \lambda)\begin{pmatrix}
0 & \overline{c}_{n}T^2(\bar{\lambda}_n) e^{-2 i \bar{\theta}_n t} \\[6pt]
0 & 0
\end{pmatrix}.
\end{array}\label{res3}
\end{equation}
For $n\in \Delta_{\lambda_{0}}^{-}$,
\begin{equation}
\begin{array}{l}
\mathop{\operatorname{Res}}\limits_{\lambda=\lambda_{n}} M_{RHP}^{(2)}(y, t, \lambda)=\lim\limits _{\lambda \rightarrow \lambda_{n}} M_{RHP}^{(2)}(y, t, \lambda){\begin{pmatrix}
0 & c_{n}^{-1}(1 / T)^{\prime}\left(\lambda_{n}\right)^{-2}e^{-2i\theta_n t} \\[6pt]
0 & 0
\end{pmatrix}}, \\[15pt]
\mathop{\operatorname{Res}}\limits_{\lambda=\overline{\lambda}_{n}} M_{RHP}^{(2)}(y, t, \lambda)=\lim\limits_{\lambda \rightarrow \overline{\lambda}_{n}} M_{RHP}^{(2)}(y, t, \lambda)\begin{pmatrix}
0 & 0 \\[6pt]
(\overline{c}_{n})^{-1}T^{\prime}\left(\bar{\lambda}_{n}\right)^{-2}e^{2 i \bar{\theta}_n t} & 0
\end{pmatrix}.
\end{array}\label{res4}
\end{equation}
\end{enumerate} 
The existence and asymptotic of $M_{RHP}^{(2)}(y, t, \lambda)$ will be shown in Section 8.
To solve $M_{RHP}^{(2)}(y, t, \lambda)$, introducing the neighborhoods of $\pm\lambda_0$
\begin{equation}
U_{\pm \lambda_{0}}=\left\{\lambda:\left|\lambda \mp \lambda_{0}\right| \leq \min \left\{\frac{\lambda_{0}}{2}, d/ 3\right\} \triangleq \varepsilon\right\}. 
\end{equation}
From the above definition and $\operatorname{dist}(\Lambda, \mathbb{R}) \geq d$ we immediately know that $\mu_{RHP}^{(2)}$ and  $\mu_{\lambda_{0}}^{(2)}$ have no poles in the neighborhoods. And the jump matrix admits the following proposition.
\begin{proposition}
The jump matrix $J^{(2)}(y, t, \lambda)$ defined by (\ref{jump4}) satisfies the estimation
\begin{equation}
\begin{aligned}
&\left\|J^{(2)}-I\right\| _{L^{\infty}\left(\Sigma_{+}^{(2)}\backslash U_{ \lambda_0}\right)}=\mathcal{O}\left(e^{-16|t| \left(|p|-\lambda_{0}\right)|q|\left(2 \lambda_{0}+|p|\right)}\right),\\
&\left\|J^{(2)}-I\right\| _{L^{\infty}\left(\Sigma_{-}^{(2)}\backslash U_{- \lambda_0}\right)}=\mathcal{O}\left(e^{-16|t| \left(|p|+\lambda_{0}\right)|q|\left(2 \lambda_{0}+|p|\right)}\right).
\end{aligned}\label{est}
\end{equation}
where $\lambda = p + iq,\quad p,q\in \mathbb{R},$ and $$\Sigma_{+}^{(2)}=\bigcup_{k=1}^{4} \Sigma_{k},\quad\Sigma_{-}^{(2)}=\bigcup_{k=5}^{8} \Sigma_{k}.$$\label{prop}
\end{proposition}
The above proposition means that the jump matrix tends to $I$ when 
$\lambda \in \Sigma_{\pm}^{(2)} \backslash U_{\pm \lambda_0}.$ Thus, as $\lambda \in \mathbb{C} \backslash U_{\pm \lambda_0},$ there is only exponential infinitesimal error while we completely ignoring the jump condition of $M_{RHP}^{(2)}(y, t, \lambda)$.
Furthermore, we break $M_{RHP}^{(2)}(y, t, \lambda)$ 
into three parts 
\begin{equation}
\mu_{RHP}^{(2)}(y, t, \lambda)=\left\{\begin{aligned}
&E(\lambda) M_{out}^{(2)}(\lambda), & \lambda \notin U_{\pm\lambda_{0}}, \\
&E(\lambda) M^{(\lambda_0)}(\lambda), & \lambda \in U_{\lambda_{0}}, \\
&E(\lambda) M^{(-\lambda_0)}(\lambda), & \lambda \in U_{-\lambda_{0}}.
\end{aligned}\right.\label{3part}
\end{equation}
Obviously, $M_{out}^{(2)}$ can be solved by a model RHP obtained by ignoring
the jump conditions of RHP5, see Section 7. For $M^{(\pm\lambda_0)}$, matching it to a parabolic cylinder model in $U_{\pm \lambda_0}$, then we can get the approximate parabolic cylinder function solution, the details see Section 8. And $E(\lambda)$ is a error function, which is a solution of a small-norm RH problem, which will be solved in Section 9.

Next step, considering the case of $\bar{\partial} \mathcal{R}^{(2)} \neq 0,$ we introduce a transformation 
\begin{equation}
M^{(3)}(y, t, \lambda)=M^{(2)}(y, t, \lambda)(M_{RHP}^{(2)})^{-1}.
\end{equation}
Thus the jump  disappeared and $M^{(3)}(y, t, \lambda)$ is continuous in $\mathbb{C}\backslash (\Sigma^{(2)}\cup \Lambda \cup \bar{\Lambda})$, we get a pure $\bar{\partial}$ problem.\\[6pt]
\textbf{RHP 6}\quad Find a matrix function $M^{(3)}(y, t, \lambda)$ with the following properties
\begin{enumerate}
    \item Analyticity: $M^{(3)}(y, t, \lambda)$ is continuous in $\mathbb{C} \backslash (\Sigma^{(2)}\cup \Lambda \cup \bar{\Lambda}).$
    \item Normalization: $M^{(3)}(y, t, \lambda) \rightarrow I$, as $\lambda \rightarrow \infty $.
    \item $\bar{\partial}$-Derivative: For $\lambda \in \mathbb{C}\backslash (\Sigma^{(2)}\cup \Lambda \cup \bar{\Lambda})$, $ \bar{\partial} M^{(3)}(\lambda)=M^{(3)}(\lambda) W^{(3)}(\lambda)$, where
    \begin{equation}
        W^{(3)}(\lambda)=M_{RHP}^{(2)}(\lambda)\bar{\partial}\mathcal{R}^{(2)}(M_{RHP}^{(2)}(\lambda))^{-1}.
    \end{equation}
\end{enumerate} 
\begin{proof}
The analyticity and normalization of $M^{(3)}(\lambda)$ can be directly derived from properties of the $M^{(2)}(\lambda)$ and $M_{RHP}^{(2)}(\lambda)$. Notice that $M^{(2)}(\lambda)$ and $M_{RHP}^{(2)}(\lambda)$ satisfy same jump condition, we have
\begin{equation}
\begin{aligned}
M_{-}^{(3)}(\lambda)^{-1} M_{+}^{(3)}(\lambda) &=M_{RHP-}^{(2)}(\lambda)M_{-}^{(2)}(\lambda)^{-1}  M_{+}^{(2)}(\lambda)M_{RHP+}^{(2)}(\lambda)^{-1} \\
&=M_{RHP-}^{(2)}(\lambda)J^{(2)}(\lambda)^{-1}M_{RHP+}^{(2)}(\lambda)^{-1}=I,
\end{aligned}
\end{equation}
which means $M^{(3)}(\lambda)$ has no jump. Additionally, we can prove that $M^{(3)}(\lambda)$ has removable singularities for $\lambda \in \Lambda\cup\bar{\Lambda},$ the method is similar as it in \cite{yang2020}. Then combine (\ref{par}) we have
\begin{equation}
\begin{aligned}
\bar{\partial}M^{(3)}&=\bar{\partial}M^{(2)}(\mu_{RHP}^{(2)})^{-1}=M^{(2)} \bar{\partial} \mathcal{R}^{(2)}(M_{RHP}^{(2)})^{-1}\\[6pt]
&=M^{(3)}[M_{RHP}^{(2)}\bar{\partial}\mathcal{R}^{(2)}(M_{RHP}^{(2)})^{-1}].
\end{aligned}
\end{equation}
\end{proof}
The pure $\bar{\partial}$ problem will be analyzed in Section 10.
\section{The solution $M_{out}^{(2)}(\lambda)$ of outer model RHP}
\hspace*{\parindent}
In this section, our aim is to establish a outer model RHP to solve $M_{out}^{(2)}(\lambda)$ and analyze the long-time behavior of soliton solutions of the initial value problem, (\ref{relation5}) indicates that we need to study the property of $M_{out}^{(2)}(\lambda)$ as $\lambda\rightarrow0$. According to Proposition \ref{prop}, we can ignore the jump at $\Sigma^{(2)}$ when $\lambda\notin U_{\pm\lambda_{0}}$. Thus we build the following outer model problem.\\[6pt]
\textbf{RHP 7}\quad Find a matrix function $M_{out}^{(2)}(y, t, \lambda)$ with the following properties
\begin{enumerate}
    \item Analyticity: $M_{out}^{(2)}(y, t, \lambda)$ is analytic in $\mathbb{C} \backslash (\Sigma^{(2)}\cup \Lambda \cup \bar{\Lambda}).$
    \item Normalization: $M_{out}^{(2)}(y, t, \lambda) \rightarrow I$, as $\lambda \rightarrow \infty $.
    \item Residues: At $\lambda=\lambda_{n}$ and $\lambda=\overline{\lambda}_{n}$, $M_{out}^{(2)}(y, t, \lambda)$ has simple poles and satisfy the same residue conditions(\ref{res3})-(\ref{res4}) with $M_{RHP}^{(2)}$.
\end{enumerate} 

Now we need to prove the existence and uniqueness of solution of the above RHP7, the idea is to discuss the reflectionless case of the RHP1 at first, and then achieve our target by replacing the scattering data. Here is the reflectionless RH problem degenerated from RHP2.\\[6pt]
\textbf{RHP 8}\quad Let $\sigma_{d}=\left\{\left(\lambda_{n}, c_{n}\right)\right\}_{n=1}^{N}$ represent the discrete data. Find a matrix function $m(\lambda;y, t|\sigma_{d})$ with the following properties
\begin{enumerate}
    \item Analyticity: $m(\lambda;y, t|\sigma_{d})$ is analytic in $\mathbb{C} \backslash (\Lambda \cup \bar{\Lambda}).$
    \item Normalization: $m(\lambda;y, t|\sigma_{d})\rightarrow I$, as $\lambda \rightarrow \infty $.
    \item Residues: At $\lambda=\lambda_{n}$ and $\lambda=\overline{\lambda}_{n}$, $m(\lambda;y, t|\sigma_{d})$ has simple poles and the residues satisfy the conditions 
    \begin{equation}
\begin{array}{l}
\mathop{\operatorname{Res}}\limits_{\lambda=\lambda_{n}} m(\lambda;y, t|\sigma_{d})=\lim\limits _{\lambda \rightarrow \lambda_{n}} m(\lambda;y, t|\sigma_{d})K_n, \\[12pt]
\mathop{\operatorname{Res}}\limits_{\lambda=\overline{\lambda}_{n}} m(\lambda;y, t|\sigma_{d})=\lim\limits_{\lambda \rightarrow \overline{\lambda}_{n}} m(\lambda;y, t|\sigma_{d})(-\sigma_{2} \overline{K_n} \sigma_{2}).
\end{array}
\end{equation}
where $K_n = \begin{pmatrix}
0 & 0 \\[6pt]
\kappa_n & 0
\end{pmatrix}, \kappa_n= c_{n}e^{2i\theta(\lambda_n)}$.
\end{enumerate} 
\begin{proposition}
The RHP8 exists an unique solution.\label{p8}
\end{proposition}
\begin{proof}
The existence can be proved in a similar way with \cite{yang2020}, and the uniqueness of solution follows from the Liouville’s theorem.
\end{proof}
The transmission coefficient satisfies trace formula $$a(\lambda)=\prod_{n=1}^{N} \frac{\lambda-\lambda_{n}}{\lambda-\bar{\lambda}_{n}}$$ under the reflectionless condition. Let $\triangle \subseteq\{1,2, \ldots, N\}$ and $\bigtriangledown= \{1,2, \cdots, N\} \backslash \triangle$. Define
\begin{equation}
a_{\triangle}(\lambda)=\prod_{n \in \triangle} \frac{\lambda-\lambda_{n}}{\lambda-\bar{\lambda}_{n}},\quad
a_{\bigtriangledown}(\lambda)=\frac{a(\lambda)}{a_{\triangle}(\lambda)}=\prod_{n \in \bigtriangledown}\frac{\lambda-\lambda_{n}}{\lambda-\bar{\lambda}_{n}}.
\end{equation}
For $\hat{M}(y, t, \lambda)$ defined by (\ref{my}), introduce a transformation
\begin{equation}
  M^{\triangle}(\lambda; y, t|\sigma^{\triangle})= \hat{M}(\lambda; y, t|\sigma_{d})a_{\triangle}(\lambda)^{\sigma_3}.\label{my2}
\end{equation}
Correspondingly, we define
\begin{equation}
    \mu^{\triangle}(\lambda; y, t|\sigma^{\triangle})=(1, 1)M^{\triangle}(\lambda; y, t|\sigma^{\triangle}).
\end{equation}
It can be directly calculated that 
\begin{equation}
\sigma^{\triangle}=\left\{\left(\lambda_{n}, c^{\prime}_{n}\right), \lambda_{n} \in \Lambda\right\}_{n=1}^{N}, \quad c^{\prime}_{n}=\left\{\begin{array}{l}
c_{n}^{-1} a_{\triangle}^{\prime}\left(\lambda_{n}\right)^{-2}, \quad n \in \triangle, \\[6pt]
c_{n} a_{\triangle}\left(\lambda_{n}\right)^{2}, \quad n \in \bigtriangledown.
\end{array}\right.\label{disd}
\end{equation}
and $M^{\triangle}(\lambda; y, t|\sigma^{\triangle})$ satisfies the modified RH problem.\\[6pt]
\textbf{RHP 9}\quad Given the discrete data shown as (\ref{disd}). Find a matrix function $M^{\triangle}(\lambda; y, t|\sigma^{\triangle})$ with the following properties
\begin{enumerate}
    \item Analyticity: $M^{\triangle}(\lambda; y, t|\sigma^{\triangle})$ is analytic in $\mathbb{C} \backslash (\Lambda \cup \bar{\Lambda}).$
    \item Normalization: $M^{\triangle}(\lambda; y, t|\sigma^{\triangle})\rightarrow I$, as $\lambda \rightarrow \infty $.
    \item Residues: At $\lambda=\lambda_{n}$ and $\lambda=\overline{\lambda}_{n}$, $M^{\triangle}(\lambda; y, t|\sigma^{\triangle})$ has simple poles and the residues satisfy the conditions 
    \begin{equation}
\begin{array}{l}
\mathop{\operatorname{Res}}\limits_{\lambda=\lambda_{n}} M^{\triangle}(\lambda; y, t|\sigma^{\triangle})=\lim\limits _{\lambda \rightarrow \lambda_{n}} M^{\triangle}(\lambda; y, t|\sigma^{\triangle})K_n^\triangle, \\[12pt]
\mathop{\operatorname{Res}}\limits_{\lambda=\overline{\lambda}_{n}} M^{\triangle}(\lambda; y, t|\sigma^{\triangle})=\lim\limits_{\lambda \rightarrow \overline{\lambda}_{n}} M^{\triangle}(\lambda; y, t|\sigma^{\triangle})(-\sigma_{2} \overline{K_n^\triangle} \sigma_{2}).
\end{array}
\end{equation}
where 
\begin{equation}
 K_n^\triangle = \left\{\begin{array}{ll}
\begin{pmatrix}
0 & \kappa_n^\triangle \\[6pt]
0 & 0
\end{pmatrix},& n\in \triangle,\\[15pt]
\begin{pmatrix}
0 & 0 \\[6pt]
\kappa_n^\triangle & 0
\end{pmatrix},& n\in \bigtriangledown,
\end{array}\right. 
\quad \kappa_n^\triangle=\left\{\begin{array}{ll}
c_{n}^{-1} a_{\triangle}^{\prime}\left(\lambda_{n}\right)^{-2} e^{-2 i t \theta\left(\lambda_{n}\right)},& n \in \triangle, \\[6pt]
c_{n} a_{\triangle}\left(\lambda_{n}\right)^{2} e^{2 i t \theta\left(\lambda_{n}\right)},& n\in \bigtriangledown.
\end{array}\right.
\end{equation}
\end{enumerate} 
It is easy to get the existence and uniqueness of the solution of RHP 9 which inherited from Proposition \ref{p8}.

Let $\triangle = \Delta_{\lambda_{0}}^{-}$ and replace the discrete data $\sigma^\triangle$ with
\begin{equation}
    \sigma^\triangle_{out}=\left\{\left(\lambda_{n}, \tilde{c}_{n}\right), \lambda_{n} \in \Lambda\right\}_{n=1}^{N},\quad \tilde{c}_{n}=\left\{\begin{array}{ll}
c^{\prime}_{n}\delta(\lambda_n)^2, & n \in \Delta_{\lambda_{0}}^{-}, \\[6pt]
c^{\prime}_{n}\delta(\lambda_n)^{-2}, & n \in \Delta_{\lambda_{0}}^{+}.
\end{array}\right.\label{disd2}
\end{equation}
Then we have the following proposition holds.
\begin{proposition}
There exists an unique solution for RHP 7 and 
\begin{equation}
  M_{out}^{(2)}(y, t, \lambda) = M^{\Delta_{\lambda_{0}}^{-}}(\lambda; y, t|\sigma^\triangle_{out}).
\end{equation}
Moreover,
\begin{equation}
  q(x,t;\sigma^\triangle_{out})=\hat{q}_{sol}(y, t ;\sigma^\triangle_{out})=\left((\mu_{out}^{(2)})_{1} (\mu_{out}^{(2)})_{2}\right)^{2}(y, t, 0)-1, 
\end{equation}
where 
\begin{equation}
    ((\mu^{(2)}_{out})_1,(\mu^{(2)}_{out})_2)(y,t,\lambda)=(1, 1)M^{(2)}_{out}(y,t,\lambda),
\end{equation}
and $q(x,t;\sigma^\triangle_{out})$ represents the $N$-soliton solution of (\ref{HDe1}) with the discrete data $\sigma^\triangle_{out}$.
\end{proposition}
Thus, for each $\triangle \subseteq\{1,2, \ldots, N\}$, there exists a solution $q(x,t;\sigma^\triangle_{out})$ of (\ref{HDe1}). For the purpose of studying the asymptotic behavior of the solution, we need to choose appropriate $\triangle$ which are well controlled as $|t|\rightarrow\infty$.

Considering the long-time behavior of $M^{\Delta_{\lambda_{0}}^{-}}(y, t, \lambda|\sigma^\triangle)$, First, we know that the one-soliton solution of (\ref{HDe1}) which has an implicit expression \cite{wm1980}
\begin{equation}
\begin{array}{c}
q(x, t)=\tanh ^{-4}\left(\kappa x-4 \kappa^{3} t+\kappa x_{0}+\varepsilon_{+}\right)-1, \\[6pt]
\varepsilon_{+}=\frac{1}{\kappa}\left[1+\tanh \left(\kappa x-4 \kappa^{3} t+\kappa x_{0}+\varepsilon_{+}\right)\right],
\end{array}
\end{equation}
where $\kappa<0$. Then define the following functions and notation which will be used later
\begin{equation}
\begin{aligned}
&I=\left\{\lambda:\frac{v_{1}}{4}<|\lambda|^{2}<\frac{ v_{2}}{4}\right\},\quad v_{1} \leq v_{2} \in \mathbb{R}^{+}, \\
&\Lambda(I)=\left\{\lambda_{n} \in \Lambda: \lambda_{n} \in I\right\}, \quad N(I)=|\Lambda(I)|, \\
&\Lambda^{-}(I)=\left\{\lambda_{n} \in \Lambda:|\lambda_n|^{2}>\frac{v_{1}}{4 }\right\}, \quad \Lambda^{+}(I)=\left\{\lambda_{n} \in \Lambda:|\lambda_n|^{2}<\frac{v_{2}}{4}\right\}, \\
&c_{n}(I)=c_{n} \mathop{\prod}\limits_{\operatorname{Re} \lambda_{k} \in I_{-} \backslash I}\left(\frac{\lambda_{n}-\lambda_{k}}{\lambda_{n}-\bar{\lambda}_{k}}\right)^{2} \exp \left[-\frac{1}{\pi i} \int_{I_{-}} \frac{\eta(s)}{s-\lambda} d s\right].
\end{aligned}
\end{equation}
And give pair points $y_1\leq y_2\in\mathbb{R},$ $v_1,v_2$ represent velocities, we introduce a cone
\begin{equation}
   C\left(y_{1}, y_{2}, v_{1}, v_{2}\right)=\left\{(y, t) \in R^{2} \mid y=y_{0}+v t, y_{0} \in\left[y_{1}, y_{2}\right], v \in\left[v_{1}, v_{2}\right]\right\}.\label{cone}
\end{equation}
Then we give two figures to show the discrete spectrum distribution and the space-time cone, see Figure \ref{f4}, \ref{f5}.
\begin{figure}[H]
 \centering
 \begin{minipage}[t]{0.35\textwidth}
  \begin{tikzpicture}[node distance=2cm]
  \fill[fill=blue!15,fill opacity=0.5] circle [radius=2];
  \fill[fill=white] circle [radius=1];
  \draw[thick ] circle [radius=1];
  \draw[thick ] circle [radius=2];
  \draw[thick] (-2.5, 0) -- (2.5, 0);
  \draw[-latex] (0,0)--(2.5, 0);
  \draw [fill] (1.3,1) circle [radius=0.05] node[left]{$\lambda_1$};
  \draw [fill] (1.3,-1) circle [radius=0.05]node[left]{$\bar{\lambda}_1$};
  \draw [fill] (-2.3,1.5) circle [radius=0.05] node[right]{$\lambda_2$};
  \draw [fill] (-2.3,-1.5) circle [radius=0.05] node[right]{$\bar{\lambda}_2$};
  \draw [fill] (-1.1,0.8) circle [radius=0.05]node[left]{$\lambda_3$};
  \draw [fill] (-1.1,-0.8) circle [radius=0.05]node[left]{$\bar{\lambda}_3$};
  \draw [fill] (0,2.3) circle [radius=0.05]node[right]{$\lambda_4$};
  \draw [fill] (0,-2.3) circle [radius=0.05]node[right]{$\bar{\lambda}_4$};
  \draw [fill] (1.6,1.7) circle [radius=0.05]node[right]{$\lambda_5$};
  \draw [fill] (1.6,-1.7) circle [radius=0.05]node[right]{$\bar{\lambda}_5$};
  \draw [fill] (1,0) circle [radius=0.05];
  \draw [fill] (2,0) circle [radius=0.05];
  \node at (2.7, 0.2) {$Re\lambda$};
  \node at (0.8,-0.3) {$\sqrt{v_1} / 2$};
  \node at (1.8,-0.3) {$\sqrt{v_2} / 2$};
  \node at (-0.5,1.4) {$I$};
  \end{tikzpicture}
\caption{Five pairs of discrete spectrum and $\Lambda(I)=\left\{\lambda_1,\lambda_3\right\}$.}
\label{f4}
\end{minipage}
\qquad
\begin{minipage}[t]{0.45\textwidth}
\begin{tikzpicture}[node distance=2cm]
  \fill[fill=blue!15,fill opacity=0.5] 
  (-1,0)--(-1.8, 2.5)--(0.9,2.5)--(1,0)--(1.8,-2.5)--(-0.9,-2.5);
  \draw[thick] (-2.5, 0) -- (2.5, 0);
  \draw[-latex] (0,0)--(2.5, 0);
  \draw[thick] (-1,0)--(-1.8, 2.5) node[left]{$y=y_1+v_2t$};
  \draw[thick] (-1,0)--(-0.9, -2.5)node[left]{$y=y_1+v_1t$};
  \draw[thick] (1,0)--(0.9, 2.5) node[right]{$y=y_2+v_1t$};
  \draw[thick] (1,0)--(1.8, -2.5)node[right]{$y=y_2+v_2t$};
  \node at (2.5, -0.3) {$y$};
  \node at (-1.2,-0.3) {$y_1$};
  \node at (1.35,-0.3) {$y_2$};
  \node at (-0.2,1.4) {$C$};
  \end{tikzpicture}
\caption{The cone $C\left(y_{1}, y_{2}, v_{1}, v_{2}\right)$.}
\label{f5}
\end{minipage}
\end{figure}

Our idea is to use the soliton solution corresponding to the discrete spectrum falling in the cone to approximate the solution  $M^{\Delta_{\lambda_{0}}^{-}}(\lambda; y, t|\sigma^\triangle)$. It yields the following proposition.
\begin{proposition}
Given discrete scattering data $\sigma^{\triangle}=\left\{\left(\lambda_{n}, c^{\prime}_{n}\right), \lambda_{n} \in \Lambda\right\}_{n=1}^{N}$, and $\sigma^{\triangle}_I=\left\{\left(\lambda_{n}, c^{\prime}_{n}\right), \lambda_{n} \in \Lambda(I)\right\}$. For $(y, t) \in C\left(y_{1}, y_{2}, v_{1}, v_{2}\right)$, as $|t|\rightarrow\infty$, 
\begin{equation}
   M^{\Delta_{\lambda_{0}}^{-}}(\lambda; y, t|\sigma^\triangle)= \left(I+\mathcal{O}\left(e^{-2 \varepsilon(I)|t|}\right)\right)M^{\Delta_{\lambda_{0}}^{-}}(\lambda; y, t|\sigma^\triangle_I),
\end{equation}
where $$\varepsilon(I)=\min_{\lambda_{n} \in \Lambda \backslash \Lambda(I)}\left\{\operatorname{Im}\left(\lambda_{n}\right) \frac{v_{1}}{|\lambda|^{2}}\left(|\lambda|+\frac{\sqrt{v_{2}}}{2}\right) \operatorname{dist}\left(\lambda_{n}, I\right)\right\}.$$
\end{proposition}
\begin{proof}
It can be proved in a similar way \cite{yang2020}.
\end{proof}

Then we can approximate the unique solution $M_{out}^{(2)}(y, t, \lambda)$ of RHP 7. The conclusion holds as follows.
\begin{proposition}
The RHP 7 exists an unique solution $M_{out}^{(2)}(y, t, \lambda)$ and define
\begin{equation}
   \left([\mu_{out}^{(2)}]_{1},[\mu_{out}^{(2)}]_{2}\right)(y, t, \lambda) =(1,1) M_{out}^{(2)}(y, t, \lambda),
\end{equation}
which satisfy
\begin{align}
M_{out}^{(2)}(y, t, \lambda) &= M^{\Delta_{\lambda_{0}}^{-}}(\lambda; y, t|\sigma^\triangle_{out})\\
&=M^{\Delta_{\lambda_{0}}^{-}}(\lambda; y, t|\sigma^\triangle_I)\mathop{\prod}\limits_{\operatorname{Re} \lambda_{k} \in I_{-} \backslash I}\left(\frac{\lambda-\lambda_{n}}{\lambda-\bar{\lambda}_{n}}\right)^{2}\delta^{-\sigma_3}+\mathcal{O}\left(e^{- \varepsilon(I)|t|}\right),
\end{align}
where $\sigma^\triangle_{out}=\left\{\lambda_{n}, \tilde{c}_{n}(\lambda_0)\right\}_{n=1}^{N}$ with
$$\tilde{c}_{n}(\lambda)=c_n\exp\left[-\frac{1}{\pi i} \int_{I_{-}} \frac{\eta(s)}{s-\lambda} d s\right].$$
Moreover, we have
\begin{align}
    q(x,t;\sigma^\triangle_{out})=\hat{q}_{sol}(y, t ;\sigma^\triangle_{out})&=\left((\mu_{out}^{(2)})_{1} (\mu_{out}^{(2)})_{2}\right)^{2}(y, t, 0)-1 \\
  &=q(x,t;\sigma^\triangle_I)+\mathcal{O}\left(e^{- \varepsilon(I)|t|}\right),
\end{align}
where $ q(x,t;\sigma^\triangle_{out})$ represents the $N$-soliton solution of (\ref{HDe1}) corresponding to the discrete scattering data $\sigma^\triangle_{out}$.
\end{proposition}
\section{The solvable RH model near phase points}
\hspace*{\parindent}
Recall that proposition \ref{prop} indicates the jump matrix tends to $I$ when $\lambda \in \Sigma_{\pm}^{(2)} \backslash U_{\pm \lambda_0}.$ As $\lambda \in U_{\pm \lambda_0},$ we consider to build a local model for function $E(\lambda)$ with a uniformly small jump.

There is no discrete spectrum in $U_{\pm \lambda_0}$, thus we replace $T(\lambda)$ with $\delta(\lambda)$ in (\ref{trans1}). Denote the expression after replacing $T_0(\pm\lambda_0)$ in the expression of $R_j(j=1,3,4,6,7,8)$ with $\delta(\pm\lambda_0)$ as $\tilde{R}_j(j=1,3,4,6,7,8)$, the jump matrix of RHP5 becomes
 \begin{equation}
J^{(2)}(y, t, \lambda)=\left\{
\begin{array}{ll}
\begin{pmatrix}
1 & 0 \\[6pt]
\tilde{R}_{1}(\lambda) e^{2 i t \theta} & 1
\end{pmatrix},& \lambda \in \Sigma_{1}, \\[15pt]
\begin{pmatrix}
1 & -\tilde{R}_{7}(\lambda) e^{-2i t \theta} \\[6pt]
0 & 1
\end{pmatrix},& \lambda \in \Sigma_{2} \cup \Sigma_{5}, \\[15pt]
\begin{pmatrix}
1 & 0 \\[6pt]
\tilde{R}_{8}(\lambda) e^{2 i t \theta} & 1
\end{pmatrix},& \lambda \in \Sigma_{3} \cup \Sigma_{8},\\[15pt]
\begin{pmatrix}
1 & -\tilde{R}_{6}(\lambda) e^{-2 i t \theta} \\[6pt]
0 & 1
\end{pmatrix}, & \lambda \in \Sigma_{4}, \\[15pt]
\begin{pmatrix}
1 & 0 \\[6pt]
\tilde{R}_{3}(\lambda) e^{2 i t \theta} & 1
\end{pmatrix},& \lambda \in \Sigma_{6},\\[15pt]
\begin{pmatrix}
1 & -\tilde{R}_{4} \left(\lambda\right) e^{2 i t \theta} \\[6pt]
0 & 1
\end{pmatrix},& \lambda \in \Sigma_{7}.
\end{array}\right.\label{jump5}
\end{equation}
According to \cite{xiao2019}, The RHP 5 can be transformed to the following solvable model.\\[6pt]
\textbf{RHP 10}\quad Find a matrix function $M^{*}(y, t, \lambda)$ with the following properties
\begin{enumerate}
    \item Analyticity: $M^{*}(y, t, \lambda)$ is meromorphic in $\mathbb{C} \backslash \Sigma^{(2)}.$
    \item Jump condition: 
    \begin{equation}
    M^{*}_{+}=M^{*}_{-} J^{(2)}(y, t, \lambda), \quad \lambda \in \Sigma^{(2)},
    \end{equation}
    where $J^{(2)}(y, t, \lambda)$ is defined as (\ref{jump5}).
    \item Normalization: $M^{*}(y, t, \lambda) \rightarrow I$, as $\lambda \rightarrow \infty $.
\end{enumerate} 
 According to proposition \ref{prop}, we get the above RH problem, which shows that jump contours can be changed from $\Sigma^{(2)}$ to $\Sigma^{(3)}$ by ignoring the jump condition on $\Sigma^{(2)}\backslash U_{\pm\lambda_0}$. The contour $\Sigma^{(3)}$ consisting of two crosses centered at $\lambda=\pm\lambda_0$, which finally leads to the asymptotics in the modulated decaying oscillations, is shown as figure \ref{f6}.
\begin{figure}[H]
 \centering
  \begin{tikzpicture}[node distance=2cm]
  \draw[dashed] (-4.5, 0) -- (4.5, 0);
  \draw[thick ](1,1)--(2,0);
  \draw[-latex](1,1)--(1.5, 0.5);
  \draw[thick ](2,0)--(3.5,-1.5);
  \draw[-latex](2,0)--(2.75,-0.75);
  \draw[thick ](-2,0)--(-1,1);
  \draw[-latex](-2,0)--(-1.5, 0.5);
  \draw[thick ](-2,0)--(-1,-1);
  \draw[-latex](-2,0)--(-1.5, -0.5);
  \draw[thick ](1,-1)--(2,0);
  \draw[-latex](1,-1)--(1.5, -0.5);
  \draw[thick ](2,0)--(3.5,1.5);
  \draw[-latex](2,0)--(2.75,0.75);
  \draw[thick ](-2,0)--(-3.5,1.5);
  \draw[-latex](-3.5,1.5)--(-2.75,0.75);
    \draw[thick ](-2,0)--(-3.5,-1.5);
  \draw[-latex](-3.5,-1.5)--(-2.75,-0.75);
  \node at (-3, 1.4) {$\Sigma_6$};
  \node at (-3, -1.4) {$\Sigma_7$};
  \node at (3, 1.4) {$\Sigma_1$};
  \node at (3, -1.4) {$\Sigma_4$};
  \node at (1.5, 0.9) {$\Sigma_2$};
  \node at (1.5,-0.9) {$\Sigma_3$};
  \node at (-1.5, 0.9) {$\Sigma_5$};
  \node at (-1.5,-0.9) {$\Sigma_8$};
  \node at (2,-0.35) {$\lambda_0$};
  \node at (-2,-0.35) {$-\lambda_0$};
  \node at (-2,-2) {$\Sigma_{-\lambda_0}$};
  \node at (2,-2) {$\Sigma_{\lambda_0}$};
  \end{tikzpicture}
\caption{The jump contour $\Sigma^{(3)}$ is composed of $ \Sigma_{-\lambda_0} $ and $ \Sigma_{\lambda_0} $.}
\label{f6}
\end{figure}
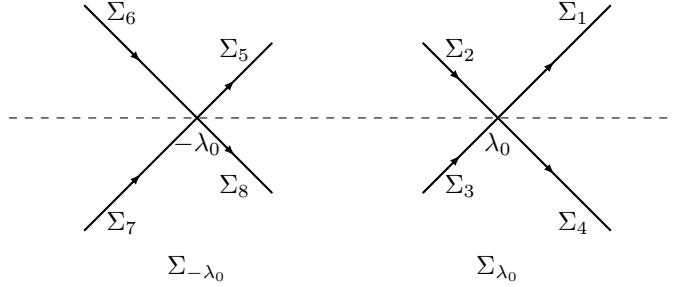
We denote $\Sigma_A$ and $\Sigma_B$ as the contours $\left\{\tilde{\lambda}=\lambda_{0} h e^{\pm i \pi / 4}:-\infty<h<\infty\right\}$, and extend the crosses $\Sigma_{-\lambda_0}$ and $\Sigma_{\lambda_0}$ to $\Sigma_A$ and $\Sigma_B$ by zero extension. We solve the RHP10 in two parts: $M^{\lambda_0}$ and $M^{-\lambda_0}$, which corresponds to the $2\times2$ matrix $M^A$ and $M^B$ respectively. Introducing the scaled operator
\begin{equation}
\left(N_{A} f\right)(\lambda)=f\left(-\lambda_{0}+\frac{\tilde{\lambda}}{\sqrt{48 t \lambda_{0}}}\right).
\end{equation}
The factor $\delta(\lambda) e^{-i t \theta(\lambda)}$ can be scaled as
\begin{equation}
N_{A} \delta(\lambda) e^{-i t \theta(\lambda)}=\delta_{A}^{0} \delta_{A}^{1}(\tilde{\lambda}),
\end{equation}
where
\begin{equation}
\begin{aligned}
\delta_{A}^{0}(\tilde{\lambda}) &=\left(192 t \lambda_{0}^{3}\right)^{i \nu / 2} e^{-8 i t \lambda_{0}^{3}} e^{\tilde{\eta}\left(-\lambda_{0}\right)}, \\
\delta_{A}^{1}(\tilde{\lambda}) &=(-\tilde{\lambda})^{-i \nu}\left(\frac{-2 \lambda_{0}}{\tilde{\lambda} / \sqrt{48 t \lambda_{0}}-2 \lambda_{0}}\right)^{-i \nu} e^{i\left(\tilde{\lambda}^{2} / 4\right)\left(1-\tilde{\lambda}\left(432 t \lambda_{0}^{3}\right)^{-1 / 2}\right)} e^{\left(\tilde{\eta}\left(\left[\tilde{\lambda} / \sqrt{48 t \lambda_{0}}\right]-\lambda_{0}\right)-\tilde{\eta}\left(-\lambda_{0}\right)\right)} .
\end{aligned}
\end{equation}
and $\tilde{\eta}\left(\lambda_{0}\right)=-\frac{1}{2 \pi i} \int_{-\lambda_{0}}^{\lambda_{0}} \log \left|\lambda_{0}-s\right| d \log \left(1-|r(s)|^{2}\right).$
Then the jump matrix $J^{\lambda_0}$ becomes $J^A$. We can decompose the new jump matrix
\begin{equation}
 J^{A}(\tilde{\lambda})=\left(I-w_{-}^{A}\right)^{-1}\left(I+w_{+}^{A}\right).  
\end{equation} 
Then the solution of the RH model centered at $-\lambda_0$ is given by 
\begin{equation}
M^{A}(\tilde{\lambda})=I+\frac{1}{2 \pi i} \int_{\Sigma_{A}} \frac{\left(\left(1-A\right)^{-1} I\right)(\zeta) w^{A}(\zeta)}{\zeta-\tilde{\lambda}} d \zeta,
\end{equation}
where $w^A=w_-^A+w_+^A.$
Particularly, the large $\tilde{\lambda}$ behavior of $M^{A}(\tilde{\lambda})$ can be expressed as 
\begin{equation}
M^{A}(\tilde{\lambda})=I-\frac{M_{1}^{A}}{\tilde{\lambda}}+O\left(\frac{1}{\tilde{\lambda}^{2}}\right).
\end{equation}

Similarly, for $\Sigma_{-\lambda_0}$, the scaled operator is introduced as 
\begin{equation}
\left(N_{B} f\right)(\lambda)=f\left(\lambda_{0}+\frac{\tilde{\lambda}}{\sqrt{48 t \lambda_{0}}}\right),
\end{equation}
and 
\begin{equation}
N_{B} \delta(\lambda) e^{-i t \theta(\lambda)}=\delta_{B}^{0} \delta_{B}^{1}(\tilde{\lambda}),
\end{equation}
where 
\begin{equation}
\begin{aligned}
\delta_{B}^{0}(\tilde{\lambda}) &=\left(192 t \lambda_{0}^{3}\right)^{-i \nu / 2} e^{8 i t \lambda_{0}^{3}} e^{\tilde{\eta}\left(\lambda_{0}\right)}, \\
\delta_{B}^{1}(\tilde{\lambda}) &=\tilde{\lambda}^{i \nu}\left(\frac{2 \lambda_{0}}{\tilde{\lambda} / \sqrt{48 t \lambda_{0}}+2 \lambda_{0}}\right)^{i \nu} e^{-i\left(\tilde{\lambda}^{2} / 4\right)\left(1+\tilde{\lambda}\left(432 t \lambda_{0}^{3}\right)^{-1 / 2}\right)} e^{\left(\tilde{\eta}\left(\left[\tilde{\lambda} / \sqrt{48 t \lambda_{0}}\right]+\lambda_{0}\right)-\tilde{\eta}\left(\lambda_{0}\right)\right)} .
\end{aligned}
\end{equation}
Particularly, the large $\tilde{\lambda}$ behavior of $M^{B}(\tilde{\lambda})$ is given by 
\begin{equation}
M^{B}(\tilde{\lambda})=I-\frac{M_{1}^{B}}{\tilde{\lambda}}+O\left(\frac{1}{\tilde{\lambda}^{2}}\right),\label{mmmm}
\end{equation}
where
\begin{equation}
\left(M_{1}^{B}\right)_{21}=-i \beta_{21},\quad\left(M_{1}^{B}\right)_{12}=i \overline{\beta_{21}}.
\end{equation}
with
\begin{equation}
\beta_{21}=\frac{r\left(\lambda_{0}\right) \Gamma(-i \nu) \nu}{\sqrt{2 \pi} e^{i \pi / 4} e^{-\pi \nu / 2}},
\end{equation}
which is obtained from \cite{xiao2019}.
With aid of symmetry (\ref{symm2}), we have $M_1^A=-\overline{M_1^B}.$

As $M^{(2)}_{out}$ is an analytic and bounded function in $U_{\pm\lambda_0}$ and use RHP5, we define two local model $M^{\pm\lambda_0}$ in (\ref{3part}) by
\begin{equation}
    M^{\pm\lambda_0}(\lambda)=M^{(2)}_{out}(\lambda)M^{*}(\lambda),\quad \lambda\in U_{\pm\lambda_0},\label{3part2}
\end{equation}
which also satisfies the jump condition of $M^{(2)}_{RHP}$.
\section{The small norm RH problem for $E(\lambda)$}
\hspace*{\parindent}
In this section, we study the small norm RH problem of $E(\lambda)$. From (\ref{3part}) and (\ref{3part2}) we have 
\begin{equation}
E(\lambda)=\left\{\begin{aligned}
&M_{R H P}^{(2)}(\lambda) M_{(o u t)}^{(2)}(\lambda)^{-1}, & \lambda \in \mathbb{C} \backslash U_{\pm \lambda_{0}}, \\
&M_{R H P}^{(2)}(\lambda) M^*(\lambda)^{-1} M_{(o u t)}^{(2)}(\lambda)^{-1}, & \lambda \in U_{\pm \lambda_{0}}.
\end{aligned}\right.\label{EEE}
\end{equation}
It is easy to get the jump contour of $E(\lambda)$ is
\begin{equation}
    \Sigma^{(E)}=\partial U_{\pm\lambda_0}\cup(\Sigma^{(2)}\backslash U_{\pm\lambda_0}),
\end{equation}
which is shown in figure \ref{f7}.
\begin{figure}[H]
 \centering
  \begin{tikzpicture}[node distance=2cm][scale=1.5]
  \draw[thick ](0,1)--(1,0);
  \draw[-latex](0,1)--(0.5, 0.5);
  \draw[thick ](1,0)--(2.5,-1.5);
  \draw[-latex](1,0)--(1.75,-0.75);
  \draw[thick ](-1,0)--(0,1);
  \draw[-latex](-1,0)--(-0.5, 0.5);
  \draw[thick ](-1,0)--(0,-1);
  \draw[-latex](-1,0)--(-0.5, -0.5);
  \draw[thick ](0,-1)--(1,0);
  \draw[-latex](0,-1)--(0.5, -0.5);
  \draw[thick ](1,0)--(2.5,1.5);
  \draw[-latex](1,0)--(1.75,0.75);
  \draw[thick ](-1,0)--(-2.5,1.5);
  \draw[-latex](-2.5,1.5)--(-1.75,0.75);
    \draw[thick ](-1,0)--(-2.5,-1.5);
  \draw[-latex](-2.5,-1.5)--(-1.75,-0.75);
  \fill[fill=white](1,0) circle [radius=0.45];
  \fill[fill=white](-1,0) circle [radius=0.45];
    \draw[thick](1,0)circle [radius=0.45];
     \draw[thick](-1,0)circle [radius=0.45];
 \draw[dashed] (-3, 0) -- (3, 0);
 \draw [fill] (1,0) circle [radius=0.03];
 \draw [fill] (-1,0) circle [radius=0.03];
 \node at (1.05,-0.2) {$\lambda_0$};
 \node at (-1.06,-0.2) {$-\lambda_0$};
  \end{tikzpicture}
\caption{The jump contour $ \Sigma^{(E)} of E(\lambda)$.}
\label{f7}
\end{figure}
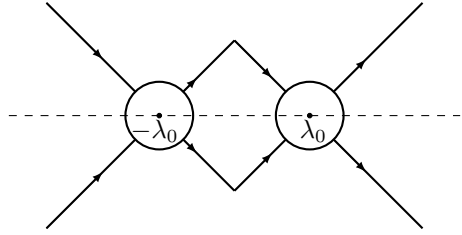
By the definition (\ref{EEE}) of $E(\lambda)$, we can calculate the jump matrix $J^{(E)}$ for $E(\lambda)$ 
\begin{equation}
J^{(E)}(\lambda)=\left\{\begin{aligned}
&M_{(o u t)}^{(2)}(\lambda) J^{(2)}(\lambda) M_{(o u t)}^{(2)}(\lambda)^{-1}, & \lambda \in \Sigma^{(2)} \backslash U_{\pm \lambda_0}, \\
&M_{(o u t)}^{(2)}(\lambda)M^{*}(\lambda) M_{(o u t)}^{(2)}(\lambda)^{-1}, & \lambda \in \partial U_{\pm \lambda_{0}} .
\end{aligned}\right.\label{jump6}
\end{equation}
Then, we establish a RH problem for $E(\lambda)$ as follows:\\[6pt]
\textbf{RHP 11}\quad Find a matrix-valued function $E(\lambda)$ with the following properties
\begin{enumerate}
    \item Analyticity: $E(\lambda)$ is meromorphic in $\mathbb{C} \backslash \Sigma^{(E)}.$
    \item Jump condition: 
    \begin{equation}
    E_{+}(\lambda)=E_{-}(\lambda) J^{(E)}(\lambda), \quad \lambda \in \Sigma^{(E)},
    \end{equation}
    where $J^{(\pm\lambda_0))}(y, t, \lambda)$ is defined as (\ref{jump6}).
    \item Normalization: $E(\lambda) \rightarrow I$, as $\lambda \rightarrow \infty $.
\end{enumerate} 

Next, we will show that for large time the function $E(\lambda)$ solves a small norm RH problem. As $M_{(o u t)}^{(2)}(\lambda)$ is bounded in $\mathbb{C}$, by applying proposition \ref{prop} and (\ref{mmmm}) , we get that as $t\rightarrow +\infty$, the jump matrix $J^{(E)}(\lambda)$ admits
\begin{equation}
\left|J^{(E)}(\lambda)-I\right|=\left\{\begin{array}{ll}
\mathcal{O}\left(e^{-16|t|\left(|p|\mp\lambda_{0}\right)|q|\left(2 \lambda_{0}+|p|\right)}\right), & \lambda \in \Sigma_{\pm}^{(2)} \backslash U_{\pm \lambda_{0}}, \\[6pt]
\mathcal{O}\left(|t|^{-1 / 2}\right), & \lambda \in \partial U_{\pm\lambda_0}.
\end{array}\right.\label{est2}
\end{equation}
Therefore, we can prove the existence and uniqueness of the RHP11 by using a small-norm RH problem, and according to Beal-Cofiman theorem, we construct the solution of RHP11. Firstly, we do a trivial decompose of $J^{(E)}$ 
\begin{equation}
    J^{(E)}=\left(b_{-}\right)^{-1} b_{+}, \quad b_{-}=I, \quad b_{+}=J^{(E)}.
\end{equation}
Then we have 
\begin{equation}
C_{E}(f)(z)=C_{-}\left(f\left(J^{(E)}-I\right)\right),\label{ef}
\end{equation}
where $C_{-}$ is the usual Cauchy projection operator on $\Sigma^{(E)}$
\begin{equation}
C_{-} f(z)=\lim _{\lambda^{\prime} \rightarrow \lambda \in \Sigma^{(E)}} \frac{1}{2 \pi i} \int_{\Sigma^{(E)}} \frac{f(s)}{s-\lambda^{\prime}} d s.
\end{equation}
Thus, the solution of RHP11 can be expressed as 
\begin{equation}
E(\lambda)=I+\frac{1}{2 \pi i} \int_{\Sigma^{(E)}} \frac{(I+\rho(s))\left(J^{(E)}-I\right)}{s-\lambda} d s,\label{edj}
\end{equation}
where $\rho \in L^{2}\left(\Sigma^{(E)}\right)$ is the unique solution of the following equation
\begin{equation}
\left(1-C_{E}\right) \rho=C_{E}(I).
\end{equation}
From (\ref{est2}) and (\ref{ef}), we get 
\begin{equation}
\left\|C_E\right\|_{L^{2}\left(\Sigma^{(E)}\right)} \lesssim\left\|C_{-}\right\|_{L^{2}\left(\Sigma^{(E)}\right)}\left\|J^{(E)}-I\right\|_{L^{\infty}\left(\Sigma^{(E)}\right)} \lesssim \mathcal{O}\left(|t|^{-1 / 2}\right),
\end{equation}
which means operator $(1-C_E)^{-1}$ exists, it follows that the existence of $\rho$ and $E(\lambda)$, the uniqueness is obvious.
In addition, we have
\begin{equation}
\|\rho\|_{L^{2}\left(\Sigma^{(E)}\right)} \lesssim \frac{\left\|C_{E}\right\|}{1-\left\|C_{E}\right\|} \lesssim|t|^{-1 / 2},
\end{equation}
which indicates the boundedness of $E(\lambda)$. 

Moreover, we consider the asymptotic behavior of $E(\lambda)$ as $\lambda\rightarrow0$ and the long-time asymptotic behavior of $E(0)$ to solve the initial value problem (\ref{HDe1}) and study the soliton resolution. As (\ref{est2})
and (\ref{edj}) indicates $E(\lambda)$ tends to zero in $\Sigma^{(E)}\backslash \partial U_{\pm\lambda_0}$ when $t\rightarrow \infty$, so we only need to consider its long time asymptotic behavior on $\partial U_{\pm\lambda_0}$.

As $\lambda\rightarrow0$, $E(\lambda)$ has expansion
\begin{equation}
E(\lambda)=E(0)+E_{1} \lambda+\mathcal{O}\left(\lambda^{2}\right),
\end{equation}
where 
\begin{equation}
E(0)=I+\frac{1}{2 \pi i} \int_{\Sigma^{(E)}} \frac{(I+\rho(s))\left(J^{(E)}-I\right)}{s} d s,
\end{equation}
\begin{equation}
E_{1}=-\frac{1}{2 \pi i} \int_{\Sigma^{(E)}} \frac{(I+\rho(s))\left(J^{(E)}-I\right)}{s^{2}} d s.
\end{equation}
And the long time asymptotic behavior can be calculated as follows,
\begin{equation}
E(0)=I+E^{(1)}|t|^{-1 / 2} +\mathcal{O}\left(|t|^{-1}\right),
\end{equation}
\begin{equation}
    E_1=E^{(2)}|t|^{-1 / 2} +\mathcal{O}\left(|t|^{-1}\right),
\end{equation}
where
\begin{equation}
\begin{aligned}
E^{(1)} &=\frac{1}{2 \pi i} \int_{\partial U_{\pm \lambda_{0}}} \frac{M_{(o u t)}^{(2)}(s)^{-1} M^B_1(\pm \lambda_0) M_{(o u t)}^{(2)}(s)}{s\left(s\mp \lambda_{0}\right)\sqrt{48\lambda_0}} d s \\[6pt]
&=-\frac{1}{\sqrt{48\lambda_{0}^3}}M_{(o u t)}^{(2)}(-\lambda_0)^{-1}M^B_1(-\lambda_0)M_{(o u t)}^{(2)}(-\lambda_0)\\[6pt]
&+\frac{1}{\sqrt{48\lambda_{0}^3}}M_{(o u t)}^{(2)}(\lambda_0)^{-1}M^B_1(\lambda_0)M_{(o u t)}^{(2)}(\lambda_0),
\end{aligned}
\end{equation}
and 
\begin{equation}
\begin{aligned}
E^{(2)} &=\frac{1}{2 \pi i} \int_{\partial U_{\pm \lambda_{0}}} \frac{M_{(o u t)}^{(2)}(s)^{-1} M^B_1(\pm \lambda_0) M_{(o u t)}^{(2)}(s)}{s^2\left(s\mp \lambda_{0}\right)\sqrt{48\lambda_0}} d s \\[6pt]
&=\frac{1}{\sqrt{48\lambda_{0}^5}}M_{(o u t)}^{(2)}(-\lambda_0)^{-1}M^B_1(-\lambda_0)M_{(o u t)}^{(2)}(-\lambda_0)\\[6pt]
&-\frac{1}{\sqrt{48\lambda_{0}^5}}M_{(o u t)}^{(2)}(\lambda_0)^{-1}M^B_1(\lambda_0)M_{(o u t)}^{(2)}(\lambda_0).
\end{aligned}
\end{equation}
Moreover, we have
\begin{equation}
E(0)^{-1}=I+\mathcal{O}\left(|t|^{-1 / 2}\right).
\end{equation}
\section{The pure $\bar{\partial}$-problem}
\hspace*{\parindent}
In this section, we study the long time asymptotics behavior of pure $\bar{\partial}$-problem RHP6. The solution of RHP6 can be given by
\begin{equation}
M^{(3)}(\lambda)=I-\frac{1}{\pi} \int_{\mathbb{C}} \frac{M^{(3)} W^{(3)}}{s-\lambda} d m(s), \label{m3}
\end{equation}
where $m(s)$ is the Lebegue measure on the $\mathbb{C}$, in fact, (\ref{m3}) is equivalent to the following expression 
\begin{equation}
M^{(3)}(\lambda)=I(I-C_{\bar{\partial}})^{-1},\label{m33}
\end{equation}
where $C_{\bar{\partial}}$ is the Cauchy integral operator
\begin{equation}
C_{\bar{\partial}}[f](\lambda)=-\frac{1}{\pi} \int_{\mathbb{C}} \frac{f(s) W(s)}{s-\lambda}d m(s).
\end{equation}

First we prove the existence of operator $C_{\bar{\partial}}$. Thus we just need to prove the following proposition.
\begin{proposition}
For $t\rightarrow\infty$, the norm of operator $C_{\bar{\partial}}$ tends to zero, and it follows that
\begin{equation}
\|C_{\bar{\partial}}\|_{L^{\infty} \rightarrow L^{\infty}} \leqslant c t^{-1 / 6}.\label{est5}
\end{equation}\label{prop2}
\end{proposition}
\begin{proof}Assume $f \in L^{\infty}(\Omega_1)$, $s=u+iv$, then we get $$Re(2 i t \theta)=8v^3t+24(\lambda_0^2-u^2)vt.$$ From (\ref{est3}) and (\ref{est4}) we have
\begin{equation}
\begin{aligned}
\left\|C_{\bar{\partial}}(f)\right\|_{L^{\infty}} & \leq\frac{1}{\pi}\|f\|_{L^{\infty}}  \int_{\Omega_1}\frac{\left|W^{(3)}(s)\right|}{|s-\lambda|} d m(s) \\[6pt]
& \leq \frac{1}{\pi}\|f\|_{L^{\infty}}\left\|M_{R H P}^{(2)}\right\|_{L^{\infty}}\left\|(M_{R H P}^{(2)})^{-1}\right\|_{L^{\infty}}\int_{\Omega_1} \frac{\left|\bar{\partial}\mathcal{R}^{(2)}(s)\right|}{|s-\lambda|} d m(s)   \\[6pt]
&\leq \frac{1}{\pi}\|f\|_{L^{\infty}}\int_{\Omega_1} \frac{\left|\bar{\partial}R_1\right|e^{8v^3t+24(\lambda_0^2-u^2)v t}}{|s-\lambda|} d m(s) \\[6pt]
& \leq c\left(I_{1}+I_{2}+I_{3}\right) ,
\end{aligned}
\end{equation}
where 
\begin{equation}
\begin{aligned}
&I_{1}=\int_{0}^{+\infty} \int_{\lambda_{0}+v}^{+\infty} \frac{\left|\bar{\partial} X_{\Lambda}(s)\right| e^{8v^3t+24(\lambda_0^2-u^2)v t}}{|s-\lambda|} d u  d v, \\[6pt]
&I_{2}=\int_{0}^{+\infty} \int_{\lambda_{0}+v}^{+\infty} \frac{\left|s_{1}^{\prime}(u)\right| e^{8v^3t+24(\lambda_0^2-u^2)v t}}{|s-\lambda|} d u d v, \\[6pt]
&I_{3}=\int_{0}^{+\infty} \int_{\lambda_{0}+v}^{+\infty}\frac{|s-\lambda_0|^{-\frac{1}{2}} e^{8v^3t+24(\lambda_0^2-u^2)v t}}{|s-\lambda|} d u  d v.
\end{aligned}
\end{equation}
To obtain (\ref{est5}), we only need to prove 
\begin{equation}
I_{j}\leq c_{j} t^{-1 / 6}, \quad j=1,2,3,
\end{equation}
where $c_j$ are different constants, the estimation is similar to \cite{yang2020}.
Hence (\ref{est5}) holds on $\Omega_1$. On other areas, the results can be similarly calculated.
\end{proof}

Next, for our purpose of studying the long time asymptotic behaviors of the solution of (\ref{HDe1}), we consider the asymptotic behavior of $M^{(3)}(0)$ and $M^{(3)}_1$
as $\lambda\rightarrow0$, $M^{(3)}(0)$ and $M^{(3)}_1$ are the different power coefficient of the following expansion
\begin{equation}
M^{(3)}(\lambda)=M^{(3)}(0)+M_{1}^{(3)} \lambda+\mathcal{O}\left(\lambda^{2}\right), \quad \lambda \rightarrow 0.
\end{equation}
The expression of $M^{(3)}(0)$ and $M^{(3)}_1$ can be easily obtained by (\ref{m3}), we have
\begin{equation}
M^{(3)}(\lambda)=I-\frac{1}{\pi} \int_{\mathbb{C}} \frac{M^{(3)} W^{(3)}}{s} d m(s), 
\end{equation}
\begin{equation}
M^{(3)}(\lambda)=\frac{1}{\pi} \int_{\mathbb{C}} \frac{M^{(3)} W^{(3)}}{s^2} d m(s). 
\end{equation}
Moreover, $M^{(3)}_1$ admits the following proposition.
\begin{proposition}
There exists constant $c$ such that
\begin{equation}
\left|M_{1}^{(3)}(y, t)\right| \leq c t^{-1},
\end{equation}
\begin{equation}
\left\|M^{(3)}(0)-I\right\| \leq c t^{-1}.
\end{equation}
\end{proposition}
\begin{proof}
The proposition can be proved in a similar way to \cite{yang2020}.
\end{proof}

\section{Soliton resolution for the Harry Dym equation}
\hspace*{\parindent}
In this section, we are ready to analyze the long time asymptotic behaviors of the soliton which solve the Harry Dym equation (\ref{HDe1}). According to the series of transformations we have done before, we have
\begin{equation}
M(x,t,\lambda)=\hat{M}(y,t,\lambda)=M^{(3)}(\lambda) E(\lambda) M^{(2)}_{out}(\lambda) \mathcal{R}^{(2)}(\lambda)^{-1} T(\lambda)^{\sigma_{3}}, \quad \lambda \in \mathbb{C} \backslash U_{\pm \lambda_{0}}.
\end{equation}
Our purpose is to reconstruct the solution $\hat{q}(y,t)$ by (\ref{relation5}), which require us to consider the situation of $\lambda\rightarrow0$. For convenience, taking $\lambda\rightarrow0$ along the imaginary axis which implies $\mathcal{R}^{(2)}(\lambda)=I$ we have
\begin{equation}
\hat{M}(y,t,\lambda)=M^{(3)}(\lambda) E(\lambda) M^{(2)}_{out}(\lambda) T(\lambda)^{\sigma_{3}},\label{solu}
\end{equation}
and then 
\begin{equation}
\begin{aligned}
\hat{M}&=\left(M^{(3)}(0)+M_{1}^{(3)} \lambda+\cdots\right)\left(E(0)+E_{1} \lambda+\cdots\right)(M^{(2)}_{out})\left(T(0)\left(1+\lambda T_{1}\right)+\cdots\right)^{\sigma_{3}}\\
&=M^{(2)}_{out}(0)T(0)^{\sigma_3}+M^{(2)}_{out}(\lambda)T(0)^{\sigma_3}T_1^{\sigma_3}\lambda+T(0)^{\sigma_3}E^{(2)}M^{(2)}_{out}(\lambda)\lambda t^{-\frac{1}{2}}+\mathcal{O}(t^{-1}).
\end{aligned}
\end{equation}
Reviewing the reconstruction formula (\ref{relation5}) of the solution $\hat{q}(y,t)$, we get
\begin{equation}
\begin{aligned}
\hat{q}(y,t)&=\left(\hat{\mu}_{1} \hat{\mu}_{2}\right)^{2}(y, t, 0)-1\\[6pt]
&=\lim _{\lambda \rightarrow 0}(\hat{M}_{11}+\hat{M}_{21})^2(\hat{M}_{12}+\hat{M}_{22})^2(y,t,\lambda)-1.
\end{aligned}
\end{equation}

Based on the above analysis, we summarize the long time asymptotic behavior of the solution as the following theorem.
\begin{thm}
Suppose that the initial values $q_0(x)\in H^{1,1}(\mathbb{R})$ satisfy the Assumption \ref{assump}, whose corresponding scattering data is recorded as $\sigma_{d}=\left\{\left(\lambda_{n}, c_{n}\right)\right\}_{n=1}^{N}$. Denote $q(x,t)$ as the solution for the initial-value problem (\ref{HDe1})-(\ref{chuzhi}).
For the fixed $y_{1}, y_{2}, v_{1}, v_{2} \in \mathbb{R}$ with $y_1\leq y_2$ and $v_1\leq v_2$, define
\begin{equation}
    I=\left\{\lambda:\frac{v_{1}}{4}<|\lambda|^{2}<\frac{ v_{2}}{4}\right\},\quad
\Lambda(I)=\left\{\lambda_{n} \in \Lambda: \lambda_{n} \in I\right\},
\end{equation}
which correspond to the solution $q(x,t;\sigma_{I})$. The corresponding scattering data are
\begin{equation}
\begin{aligned}
&\sigma_{I}=\left\{\left(\lambda_{n}, c_{n}(I)\right), \lambda_{n} \in \Lambda(I)\right\},\\
&c_{n}(I)=c_{n} \prod_{\operatorname{Re} \lambda_{k} \in I_{-} \backslash I}\left(\frac{\lambda_{n}-\lambda_{k}}{\lambda_{n}-\bar{\lambda}_{k}}\right)^{2} \exp \left[-\frac{1}{\pi i} \int_{I_{-}} \frac{\eta(s)}{s-\lambda} d s\right].
\end{aligned}
\end{equation}
Then for $(y,t)\in C(y_1,y_2,v_1,v_2)$ and $t\rightarrow \infty$, we have
\begin{equation}
\begin{aligned}
 &q(x,t)=\hat{q}(y,t)=\hat{q}_{sol}(y,t;\sigma_I)+m_1^2m_2^2-1,\\
 &y=x(y,t)+\frac{1}{2i}(\frac{m_3-m_4}{m_5-m_6}-1)+\mathcal{O}(t^{-1}),
\end{aligned}\label{soliton}
\end{equation}
where $C(y_1,y_2,v_1,v_2)$ defined by (\ref{cone}) and 
\begin{equation}
\begin{aligned}
   & m_1=[M^{(2)}_{out}(0)]_{11}T(0)-[M^{(2)}_{out}(0)]_{21}T(0),\\
   & m_2=[M^{(2)}_{out}(0)]_{12}T(0)-[M^{(2)}_{out}(0)]_{22}T(0),\\
& m_3=[M^{(2)}_{out}(0)]_{11}T(0)T_1+[E^{(2)}M^{(2)}_{out}(0)]_{11}T(0)t^{-\frac{1}{2}},\\
    & m_4=[M^{(2)}_{out}(0)]_{21}T(0)T_1+[E^{(2)}M^{(2)}_{out}(0)]_{21}T(0)t^{-\frac{1}{2}},\\
   & m_5=[M^{(2)}_{out}(0)]_{12}T(0)T_1+[E^{(2)}M^{(2)}_{out}(0)]_{12}T(0)t^{-\frac{1}{2}},\\
  &  m_6=[M^{(2)}_{out}(0)]_{22}T(0)T_1+[E^{(2)}M^{(2)}_{out}(0)]_{22}T(0)t^{-\frac{1}{2}}.
  \end{aligned}
\end{equation}
\end{thm}

So far, the long time asymptotic expansion (\ref{soliton}) shows the soliton resolution of for the initial value problem of the Harry Dym equation which contains the soliton term by $N(I)$-soliton whose parameters are modulated by
a sum of localized soliton-soliton interactions as one moves through the cone and the second term coming from soliton-radiation interactions on continuous spectrum up to an residual error of order $\mathcal{O}(t^{-1})$ from the $\bar{\partial}$ equation.

\section*{Acknowledgements}
\hspace*{\parindent}
This work is sponsored by the National Natural Science Foundation of China (No. 11571079), Shanghai Pujiang Program (No. 14PJD007) and the Natural Science Foundation of Shanghai (No. 14ZR1403500), and the Young Teachers Foundation (No. 1411018) of Fudan university.

\bibliographystyle{unsrt}
\bibliography{huawei}

\end{document}